\documentclass[11pt,dvipsnames]{article}

\usepackage[utf8]{inputenc}
\usepackage{geometry,authblk}
\usepackage[table,xcdraw]{xcolor}
\usepackage[utf8]{inputenc}
\usepackage{graphicx} 
\usepackage{bussproofs}
\usepackage{amsmath,amssymb,amsfonts}
\usepackage{mathrsfs,latexsym,amsthm,enumerate}
\usepackage{tikz}
\usepackage{tikz-cd}
\usetikzlibrary{cd}
\usepackage{cite}
\usepackage{amscd}
\usepackage[all]{xy}
\usepackage[utf8]{inputenc}
\usepackage{amsmath,amssymb,amsfonts}
\usepackage{mathrsfs,latexsym,amsthm,enumerate}
\usepackage{breqn}
\usepackage{fancyhdr}
\usepackage{sectsty}
\usepackage{stmaryrd}
\usepackage{url}
\usepackage{bbm}
\usepackage{subcaption}
\usepackage{float}
\usepackage{xcolor}

\theoremstyle{plain}
\newtheorem{lemma}{Lemma}[section]
\newtheorem{proposition}[lemma]{Proposition}
\newtheorem{corollary}[lemma]{Corollary}

\newtheorem{theorem}[lemma]{Theorem}
\newtheorem{observation}[lemma]{Observation}

\theoremstyle{definition}
\newtheorem{remark}[lemma]{Remark}

\newcommand{\ra}{\rightarrow}
\newcommand{\epi}{\twoheadrightarrow}
\newcommand{\inclu}{\hookrightarrow}

\newcommand{\up}{{\uparrow}}

\newcommand{\tc}{\textit}

\newcommand{\ca}{\mathcal}
\newcommand{\mf}{\mathsf}

\newcommand{\se}{\subseteq}
\newcommand{\sm}{\setminus}

\newcommand{\we}{\wedge}
\newcommand{\ve}{\vee}
\newcommand{\bwe}{\bigwedge}
\newcommand{\bve}{\bigvee}

\newcommand{\bcu}{\bigcup}
\newcommand{\SSS}{\mf{S}}

\newcommand{\bl}{\mathfrak{b}}
\newcommand{\cl}{\mathfrak{c}}
\newcommand{\op}{\mathfrak{o}}
\newcommand{\spa}{\mathsf{sp}}
\newcommand{\sob}{\mathsf{sob}}

\newcommand{\Om}{\Omega}
\newcommand{\pt}{\mathsf{pt}}

\newcommand{\SL}{\mathsf{S}(L)}

\newcommand{\SLb}{\mf{S}_b(L)}
\newcommand{\SLd}{\mf{S}_D(L)}
\newcommand{\SLc}{\mf{S}_{c}(L)}
\newcommand{\set}[1]{\{\,#1\,\}}
\newcommand{\bigset}[1]{\bigl\{\,#1\,\bigr\}}
\newcommand{\cat}[1]{\ensuremath{\mathsf{#1}}} 
\newcommand{\donotbreakdash}[1]{#1\nobreakdash-\hspace{0pt}}

\title{The coframe of \donotbreakdash{$D$}sublocales of a locale and the \donotbreakdash{$T_D$}duality}

\date{}

\author[1,2]{Igor Arrieta}
\author[1,3]{Anna Laura Suarez}
\affil[1]{University of Coimbra, CMUC, Department of Mathematics}
\affil[2]{Universidad del Pa\'is Vasco UPV/EHU, Departamento de Matem\'aticas}
\affil[3]{University of Birmingham, School of Computer Science}

\begin{document}

\maketitle

\abstract{The notion of \emph{D-sublocale} is explored. This is the notion analogue to that of sublocale in the duality of \donotbreakdash{$T_D$}spaces. A sublocale $S$ of a frame $L$ is a \donotbreakdash{$D$}sublocale if and only if the corresponding localic map preserves the property of being a covered prime. It is shown that for a frame $L$ the system of those sublocales which are also \donotbreakdash{$D$}sublocales form a dense sublocale $\SLd$ of the coframe $\SL$ of all its sublocales. It is also shown that the spatialization $\spa_D[\SLd]$ of $\SLd$ consists precisely of those \donotbreakdash{$D$}sublocales of $L$ which are \donotbreakdash{$T_D$}spatial. Additionally, frames such that we have $\SLd\cong \ca{P}(\pt_D(L))$ --- that is, those such that \donotbreakdash{$D$}sublocales perfectly represent subspaces --- are characterized as those \donotbreakdash{$T_D$}spatial frames such that $\SLd$ is the Booleanization of $\SL$. \\[1mm]
}

\section*{Introduction}
For a point-free space (that is, a frame) $L$ the ordered collection of its point-free subspaces, that is, its \emph{sublocales}, is always a coframe usually denoted as $\SL$. Embedded in $\SL$ as a coframe we have the coframe $\mf{Cl}(L)$ of the \emph{closed} sublocales of $L$, point-free versions of the closed subspaces of a space. We have an anti-isomorphism $L\ra \mf{Cl}(L)$. Embedded in $\SL$ as a frame we also have $\mf{Op}(L)$, the ordered collection of all \emph{open} sublocales of $L$; these represent the open subspaces of the space $L$. It is not surprising, then, that the frame $\mf{Op}(L)$ is isomorphic to $L$. Since the system $\SL^{op}$ is a frame, it has its own sublocales, and among these there are certain special ones that encode information about the frame $L$. Next to the sublocales of $\SL^{op}$ that are particularly interesting are also certain subsets. Recently, some of these special sublocales and subsets have enjoyed special attention. We have the system of all joins of closed sublocales $\SLc$ (see \cite{joinsofclosed19}), which coincides with the smallest dense sublocale of $\SL$ if and only if the frame $L$ satisfies a weak separation axiom called \emph{subfitness}. We then have $\SLb$, the smallest dense sublocale of $\SL$. The properties of this sublocale have been studied in \cite{Arrieta20}. We also have $\spa[\SL]$ (see \cite{Suarez20}), the sublocale of $\SL$ consisting of those sublocales which are spatial --- that is, those sublocales which are the frame of opens of some topological space. 

We explore the relation these important subsets of the assembly $\SL$ of a frame $L$; and we identify another which also encodes some special properties of $L$, in particular in the context of the \donotbreakdash{$T_D$}duality. We will look at the notion of \tc{\donotbreakdash{$D$}sublocale} of $L$, a sublocale $S\se L$ such that the corresponding frame surjection is a $D$-homomorphism (i.e. it is such that every prime covered in $S$ is also covered in $L$). We take \donotbreakdash{$D$}sublocales as being the analogues of sublocales in the \donotbreakdash{$T_D$}duality. We show that for a frame $L$ the \donotbreakdash{$D$}sublocales form a coframe $\SLd$ which is always a ($D$-)subcolocale of $\SL$.

We then relate the subcolocale $\SLd$ with the subcolocale $\spa[\SL]$ (i.e. the ordered collection of all spatial sublocales of $L$); as well as with the subcolocale $\SLb$ (i.e. the Booleanization subcolocale of $\SL$). We also relate all these structures with the subset $\SLc\se \SL$ of all joins of closed sublocales. We obtain characterizations of all possible set inclusions between these structures, as depicted in the table below. 

\begin{table}[H]
\begin{center}
 \begin{tabular}{|c|c|}
\hline
\rowcolor[HTML]{EFEFEF} 
\textbf{Relations between sublocales of $\SSS(L)$} & \textbf{Properties of $L$} \\ \hline
              $\SSS_b(L)\subseteq \spa[\SSS(L)]$                                &           Spatial                 \\ \hline
              $\SSS_b(L)= \spa[\SSS(L)]$                              &                        Strongly \donotbreakdash{$T_D$}spatial     \\ \hline
              $\SSS_b(L)=\SSS(L)$                                 &             Scattered               \\ \hline
              $\SSS(L)=\SSS_D(L)$                                 &        Primes are covered                    \\ \hline
              $\SSS(L)=\spa(\SSS(L))$                                 &        Totally spatial                    \\ \hline
              $\SSS_b(L)=\SSS_D(L)$                             &        \donotbreakdash{$D$}scattered                    \\ \hline
                            $\SSS_D(L)\subseteq \SSS_b(L)$      &        \donotbreakdash{$D$}scattered                                          \\ \hline
                 $\SSS_D(L)\subseteq \spa(\SSS(L))$                              &      Totally spatial                       \\ \hline
                         $\spa[\SSS(L)]     \subseteq \SSS_D(L)$            &        Primes are covered                    \\ \hline
                                      
              $\SLc\se \spa[\SL]$  &  Spatial \\ \hline
               
              $\spa[\SL]\se \SLc$  &  Primes are maximal\\ \hline
              
              $\SLd\se \SLc$ &      Subfit + \donotbreakdash{$D$}scattered    \\ \hline
              
              $\SLd=\SLc$ & Subfit + \donotbreakdash{$D$}scattered \\ \hline
              $\SL= \SLc$ & Subfit + scattered\\ \hline
\end{tabular}
\end{center}
\end{table}

 We then show that we have an adjunction of posets
 \[
\begin{tikzcd}[row sep=huge, column sep=huge, text height=1.5ex, text depth=0.25ex]
 \mathcal{P}(\pt_D(L))  \arrow[bend left=20]{r}[name=a]{\mathfrak{M}} & \mf{S}_D(L) \arrow[bend left=20]{l}[name=b]{\pt_D} \arrow[phantom, from=a, to=b]{}{\ \bot}
\end{tikzcd}
\]
Here, $\pt_D(S)$ is the collection of covered primes of a \donotbreakdash{$D$}sublocale $S$. As $S$ is assumed to be a \donotbreakdash{$D$}sublocale, there is no ambiguity in this phrase as primes of $S$ are covered in $S$ if and only if they are covered in $L$. The adjunct $\pt_D$, then, takes the $T_D$ spectrum of a \donotbreakdash{$D$}sublocale, while the adjunct $\mathfrak{M}$ closes the subset under arbitrary meets. The fixpoints of $\pt_D\circ \mathfrak{M}$, then, are the equivalents of the sober spaces in the \donotbreakdash{$T_D$}duality (those isomorphic to the spectrum of their frame of opens); while the fixpoints of $\mathfrak{M}\circ \pt_D$ are those \donotbreakdash{$D$}sublocales which play the role of spatial frames in the classical duality (i.e. they are those isomorphic to the frame of opens of their spectrum).

It is shown that the composition $\pt_D\circ\mathfrak{M}$ is the identity. It is shown that the fixpoints of $\mathfrak{M}\circ \pt_D$ are the spatial \donotbreakdash{$D$}sublocales. We finally use this adjunction to ask ourselves how we can characterize those frames $L$ such that the \donotbreakdash{$D$}sublocales of $L$ perfectly represent the subspaces of $\pt_D(L)$. We reach the conclusion that there is a fundamental difference in the way subspaces and sublocales behave in the \donotbreakdash{$T_D$}duality, compared to the classical duality. In the classical spatial-sober duality, we could ask ourselves three questions. In the following $\sob[\ca{P}(\pt(L))]$ denotes the ordered collection of the sober subspaces of $\pt(L)$.
\begin{enumerate}
    \item When is it that for a frame $L$ we have that taking the spectrum induces an isomorphism $\spa[\SL]\cong \ca{P}(\pt(L))$? These are exactly the frames for which all subspaces of $\pt(L)$ are sober.
    \item When is it that for a frame $L$ we have that taking the spectrum induces an isomorphism $\SL\cong \sob[\ca{P}(\pt(L))]$? These are exactly the frames for which all sublocales of $L$ are spatial (the \tc{totally spatial} frames, explored for instance in \cite{niefield87}).
  \item When is it that we have both of the above conditions? In other words, when is it that taking the spectrum induces an isomorphism $\SL\cong \ca{P}(\pt(L))$? Adapting a famous result by Simmons (see \cite{simmons80}) we find that these are the frames such that they are spatial and the space $\pt(L)$ is \tc{scattered}.
\end{enumerate}
In the \donotbreakdash{$T_D$}duality, the answer to the analogue of the first question collapses to ``always", and so the second and the third question collapse to the same.

\tableofcontents

\section*{Preliminaries}
We first recall some background on point-free topology. For more information on the categories of frames and locales, we refer the reader to Johnstone \cite{johnstone82} or the more recent Picado-Pultr \cite{picadopultr2011frames}. A  \emph{locale} (or \emph{frame}) is a complete lattice $L$ satisfying
$$a\wedge \bigvee B = \bigvee \set{a\wedge b\mid b\in B}$$
for all $a\in L$ and $B\subseteq L$. A \emph{frame homomorphism} is a function preserving arbitrary joins (including the bottom element $0$) and finite meets (including the top element $1$). Frames and their homomorphisms form a category \cat{Frm}. For each $a$ the map $a\wedge(-)$ preserves arbitrary joins, thus it has a right (Galois) adjoint $a \to (-)$,  making $L$ a complete Heyting algebra (i.e. a cartesian closed category, if one regards $L$ as a  thin category). This right adjoint is called the \emph{Heyting operator}. In particular, the \emph{pseudocomplement} of an $a\in L$ is the element $a^*=a\to 0$ and it can be characterized as the largest $b\in L$ such that $a\wedge b=0$.

\subsection*{The categories \cat{Loc} and \cat{Frm}} Given a topological space $X$, its lattice of open sets $\Omega(X)$ is always a frame, and this construction can be upgraded to a functor $\Omega\colon {\cat{Top}\longrightarrow \cat{Frm}^{op}}$ which sends a continuous map $f\colon X\to Y$ to the preimage operator $f^{-1}[-]\colon \Omega(Y)\to\Omega(X)$. An element $p$ of a frame $L$ is said to be \emph{prime} if whenever $p=x\wedge y$ for some $x,y\in L$, then $p=x$ or $p=y$. The collection of all primes of $L$ will be denoted by $\pt(L)$ and we shall refer to it as the \emph{spectrum} (or the \emph{classical spectrum} or \emph{sober spectrum}) of $L$. For each $a\in L$, we set $\Sigma_a=\set{ p\in \pt(L)\mid a\not\leq p} $. Then the family $\set{ \Sigma_a\mid a\in L} $ is a topology on $\pt(L)$.  The assignment $\Sigma(L)=(\pt(L),\set{\Sigma_a\mid a\in L})$  extends to an spectrum functor $\Sigma\colon {\cat{Frm}^{op}\longrightarrow \cat{Top}}$ which yields an adjunction 
\[
\begin{tikzcd}[row sep=huge, column sep=huge, text height=1.5ex, text depth=0.25ex]
 \cat{Top} \arrow[bend left=20]{r}[name=a]{\Omega} &\cat{Frm}^{op}  \arrow[bend left=20]{l}[name=b]{\Sigma} \arrow[phantom, from=a, to=b]{}{\ \bot}
\end{tikzcd}
\]

The category $\cat{Loc}$ of locales is by definition the opposite category of $\cat{Frm}$: $$\cat{Loc}=\cat{Frm}^{op},$$
and $\Omega$ restricts to a full embedding of a substantial part of $\cat{Top}$ (namely the full subcategory of sober spaces) into $\cat{Loc}$. The latter can therefore be seen as a category of generalized spaces. We shall mostly speak of objects in \cat{Loc} as locales (instead of frames) when emphasizing the covariant approach. Morphisms in \cat{Loc} can be concretely represented by the right (Galois) adjoints $f_*\colon M\longrightarrow L$ of the corresponding frame homomorphisms $f\colon L\longrightarrow M$; these will be referred to as \emph{localic maps}.

A frame is said to be \emph{spatial} if $L\cong \Omega(X)$ for some space $X$, or equivalently if $L\cong \Omega(\Sigma(L))=\set{\Sigma_a\mid a\in L}$. For any frame $L$, there is a surjective frame homomorphism  (the counit of the adjunction) $L\twoheadrightarrow \Omega(\Sigma(L))$ sending $a$ to $\Sigma_a$; this map  is usually called the \emph{spatialization} of $L$. 

A topological space is \emph{sober} if every prime of $\Omega(X)$ is of the form $X\setminus \overline{\{x\}}$ for a unique $x\in X$. It is well-known that the space $\pt(L)$ defined above is sober for any frame $L$.
\subsection*{Sublocales}\label{presub} A regular subobject in $\cat{Loc}$ (that is, an isomorphism class of regular monomorphisms) of a locale $L$ is \emph{a sublocale} of $L$. Sublocales of a locale $L$ can be represented as the actual subsets $S\subseteq L$ such that 
\begin{enumerate}
\item \label{subp1} $S$ is closed under arbitrary meets in $L$, and 
\item \label{subp2} $a\to s\in S$ for all $a\in L$ and $s\in S$.  \end{enumerate}
A different, but equivalent, representation of sublocales is by means of nuclei, i.e. inflationary and idempotent maps $\nu\colon L\longrightarrow L$ which preserve binary meets. The sublocale associated to a nucleus $\nu$ is the image $\nu[L]$, and conversely the nucleus associated to a sublocale $S\subseteq L$ is given by $\iota_S \circ \nu_S$  where $\iota_S$ denotes the inclusion of $S$ into $L$ and $\nu_S$ is its left adjoint frame homomorphism given by 
$$\nu_S(a)=\bigwedge \set{ s\in S \mid s\geq a }.$$ 
A locale $L$ is said to be \emph{totally spatial} if each of its sublocales is spatial.

A sublocale should not be confused with a \emph{subframe}; the latter is a subobject of a frame in the category $\cat{Frm}$. Subframes can be represented as subsets which are closed under arbitrary joins and finite meets.

\subsubsection*{Closed and open sublocales}\label{opclsub} For each $a\in L$, one has an \emph{open sublocale} and a \emph{closed sublocale}$$\mathfrak{o}(a)=\set{b\mid b=a\to b}=\set{a\to b\mid b\in L} \quad\textrm{and}\quad \mathfrak{c}(a)=\uparrow a$$
which in the spatial case $L=\Omega(X)$ correspond to the open and closed subspaces. 

If $S$ is a sublocale of $L$, the \emph{closure} of $S$ in $L$, denoted by $\overline{S}$, is the smallest closed sublocale containing $S$, which can be computed as $\overline{S}=\mathfrak{c}(\bigwedge S)$. 
A sublocale $S$ is \emph{dense} if $\overline{S}=L$, or equivalently if $0\in S$. A sublocale $S$ is said to be codense if $\nu_S(a)=1$ implies $a=1$.

\subsubsection*{Boolean sublocales} For each $a\in L$, there is a sublocale $\mathfrak{b}(a)=\set{b\to a\mid b\in L}$ which turns out to be Boolean. Moreover, every Boolean sublocale of $L$ is of the form $\mathfrak{b}(a)$ for some $a\in L$. If $p\in L$ is a prime, it is easy to check that then $\mathfrak{b}(p)=\set{1,p}$, and these sublocales are usually referred to as \emph{one-point sublocales}. 

In particular, the sublocale $\mathfrak{b}(0)=\set{x^*\mid x\in L}$ is called the \emph{Booleanization} of $L$ and it can be characterized either as the least dense sublocale of $L$ or the unique Boolean dense sublocale of $L$.

\subsubsection*{The coframe of sublocales}\label{univ} The family $$\mathsf{S}(L)$$ of all sublocales of $L$ partially ordered by inclusion is a coframe (i.e. its dual poset $\SSS(L)^{op}$ is a frame) and lattice operations in $\SSS(L)$ are given by
$$\bigwedge_{i\in I} S_i =\bigcap_{i\in I} S_i ,\quad \bigvee_{i\in I} S_i= \bigset{\bigwedge A\mid A\subseteq \bigcup_{i\in I} S_i}.$$
It is not generally the case that $\SSS(L)$ is Boolean but it still has plenty of complemented sublocales: the frame $\SSS(L)^{op}$ is always zero-dimensional.  More precisely, every sublocale $S$ of $L$ can be written $S=\bigcap_i\mathfrak{c}(a_i)\vee \mathfrak{o}(b_i)$ for some $\{a_i\}_{i\in I}$, $\{b_i\}_{\in L}$, and  closed sublocales $\cl(a)$ and open sublocales $\op(a)$ are  complements of each other. We also have
$$\cl\bigl(\bigvee_i a_i\bigr)=\bigcap_i \cl(a_i),\quad \op\bigl(\bigvee_i a_i\bigr)=\bigvee_i \op(a_i),$$
$$\cl( a\wedge b)= \cl(a)\vee \cl(b),\quad \op( a\wedge b)= \op(a)\cap \op(b).$$
Further, a sublocale is said to be \emph{locally closed} if it is of the form $\op(a)\cap \cl(b)$ for some $a,b\in L$.

Since $\SSS(L)$ is a coframe, there is a \emph{co-Heyting operator} giving the \emph{difference} $S\smallsetminus T$ of two sublocales $S,T\in \SSS(L)$; this operator is characterized by the condition $$S\smallsetminus T \subseteq R \iff S\subseteq T\vee R.$$
In particular, the \emph{supplement} of $S\in\SSS(L)$ is $S^{\#}= L\smallsetminus S$, i.e. the smallest sublocale of $L$ whose join with $S$ is $L$. We shall freely use some its properties, e.g. the ones listed below
\begin{enumerate}
\item $S\setminus T\subseteq S$;
\item $S\setminus T=0$ iff $S\subseteq T$;
\item $S\setminus C=S\cap C^{\#}$;
\item $S\setminus \bigcap_i S_i=\bigvee_i (S\setminus S_i)$; 
\item $(S\setminus T)\setminus R = (S\setminus R)\setminus T$;
\end{enumerate}
for each $S,T,R\in \SSS(L)$, $\{S_i\}_{i\in I}\subseteq \SSS(L)$ and complemented sublocale $C\subseteq L$. A comprehensive list of its properties may be found in \cite{remainders}.

For several notions related to coframes, we shall use terminology from frame theory modified just by adding the prefix ``co-''. For example, a subset $S$ of a coframe $L$ will be said to be a \emph{subcolocale}  of $L$ if  $S^{op}$ is a sublocale of the frame $L^{op}$ (i.e. iff it is a subset of $L$ closed under arbitrary joins and the co-Heyting operator). Sometimes we will omit the prefix ``co-'' in order to avoid confusion: we will speak of a dense (resp. spatial, \donotbreakdash{$T_D$}spatial) subcolocale $S$ of a coframe $L$ for meaning that $S^{op}$ is a dense (resp. spatial, \donotbreakdash{$T_D$}spatial) sublocale of $L^{op}$.

\subsubsection*{Images and preimages} Every localic map $f\colon L\longrightarrow M$ gives, by pulling back in $\cat{Loc}$, an \emph{inverse image} map \linebreak$f_{-1}[-]\colon \mf{S}(M)\longrightarrow \SL$ which turns out to be a coframe homomorphism (i.e. a function which preserves arbitrary meets and finite joins). Set-theoretic direct image yields a \emph{direct image} map
$f[-] \colon \SL\longrightarrow \mf{S}(M)$ which is additionally a colocalic map (i.e. a left adjoint of a coframe homomorphism). In this context, the usual adjunction $f[-] \dashv f_{-1}[-]$ is satisfied.

\subsubsection*{Joins of closed (resp. complemented) sublocales}  Let $\SSS_c(L)$ denote the subset of $\SSS(L)$ consisting of joins of closed sublocales i.e.
$$\SSS_c(L)=\bigset{ \bigvee_{a\in A} \cl(a)\mid A \subseteq L},$$
 endowed with the inclusion order inherited from $\SSS(L)$. In the recent paper \cite{joinsofclosed19}, Picado, Pultr and Tozzi  show (a.o.) that $\SSS_c(L)$ is \emph{always} a frame which is embedded as a join-sublattice in the coframe $\SSS(L)$. One of the main results from \cite{joinsofclosed19} is that 
\begin{quote}\emph{if $L$ is subfit, and only in that case, $\SSS_c(L)$ is a Boolean algebra and coincides precisely with the Booleanization of $\SSS(L)$}.
\end{quote}
More generally, if $L$ is not necessarily subfit, it is also of interest to consider the Booleanization of $\SSS(L)$, denoted by $\SSS_b(L)$. The latter is precisely the system consisting \emph{smooth} sublocales, that is, those which are joins of complemented sublocales (or equivalently joins of locally closed sublocales). A study of the properties of $\SSS_b(L)$ may be found in \cite{Arrieta20}.
 
 \subsubsection*{Subspaces and sublocales}
 
 In \cite{Suarez20}, the relation between sublocales of a frame $L$ and subspaces of its spectrum $\pt(L)$ is explored. Spatial sublocales and sober subspaces are seen as fixpoints of the following adjunction of posets.
\[
\begin{tikzcd}[row sep=huge, column sep=huge, text height=1.5ex, text depth=0.25ex]
 \mathcal{P}(\pt(L))  \arrow[bend left=20]{r}[name=a]{\mathfrak{M}} & \mf{S}(L) \arrow[bend left=20]{l}[name=b]{\pt} \arrow[phantom, from=a, to=b]{}{\ \bot}
\end{tikzcd}
\]
Here, $\pt\colon\SL\ra \ca{P}(\pt(L))$ maps each sublocale to the collection of primes contained in it, while $\mathfrak{M}\colon\ca{P}(\pt(L))\ra \SL$ takes a collection of primes of $L$ and it closes it under arbitrary meets. For every frame $L$, we have that its spatialization is a surjection, and so it corresponds to a certain sublocale that we will denote by $\spa(L)$. For a sublocale $S\se L$, its spatialization sublocale is $\mathfrak{M}(\pt(S))$. Hence, the composition $\mathfrak{M}\circ \pt$ of the two adjoints above is spatialization. Similarly, the composition $\pt\circ \mathfrak{M}$ is sobrification, up to homeomorphism. For every frame $L$ we have the following commuting diagram in the category of coframes. The bottom arrow sends a sublocale $S$ to its prime spectrum $\pt(S)$, which simply consists of the collection $\pt(L)\cap S$. 
 
\begin{center}
\begin{tikzcd}[row sep=large, column sep = large]
\spa [\mf{S}(L)]  
\arrow{r}{\pt(\cong)}  
& \sob[\ca{P}(\pt(L))] 
\arrow[hookrightarrow]{d}  \\
\mf{S}(L) 
\arrow{r}{\pt} 
\arrow[twoheadrightarrow]{u}{\spa} 
& \ca{P}(\pt(L))  
\end{tikzcd}
\end{center}

The coframe $\spa[\SL]$ denotes the ordered collection of spatial sublocales of $L$, while $\sob[\ca{P}(\pt(L))]$ denotes the ordered collection of sober subspaces of $\pt(L)$. Finally, the vertical arrow maps each sublocale $S$ to its spatialization $\mathfrak{M}(\pt(S))$. 

\subsection*{The axiom $T_D$}\label{tdsubsection}

 A space $X$ is defined to be $T_D$ if every point $x\in X$ has some neighborhood $U$ with $U{\sm}\{x\}$ open. This is an axiom stronger than $T_0$ and weaker than $T_1$ and it was introduced in \cite{Aull62}. It has been used (for instance in \cite{Pultr94}) in order to answer the question of when a topological space can be completely recovered from its frame of opens. The \textit{Skula space} of a space $X$, denoted as $Sk(X)$, is the space defined as follows. The set of points is the same set of points as $X$, while the topology is the \textit{Skula topology}, that is the one generated by the opens of $X$ together with their complements. The following characterizations of \donotbreakdash{$T_D$}spaces can be found in \cite{picadopultr2011frames}.
 
 \begin{proposition}
 The following are equivalent for a $T_0$ space $X$.
 \begin{enumerate}
     \item The space is $T_D$.
     \item For no $x\in X$ do we have that the dualization of the inclusion $X{\sm}\{x\}\se X$ is a frame isomorphism.
     \item If $Y\nsubseteq Z$ then $\Om'(Y)\nsubseteq \Om'(Z)$, for all subspaces $Y,Z\se X$.
     \item The Skula space $Sk(X)$ is discrete.

 \end{enumerate}
 \end{proposition}

One can exhibit the role of the axiom $T_D$ in point-free topology by comparing it to sobriety. This has been done in detail in \cite{banaschewskitd}. In particular, in the context of $T_0$ spaces, the two axioms mirror each other in that sober spaces are maximal in the same sense in which \donotbreakdash{$T_D$}spaces are minimal.

\begin{itemize}
    \item A space $X$ is sober if and only if we can never have a nontrivial subspace inclusion $X\inclu Y$ such that its dualization is an isomorphism.
    \item A space $X$ is $T_D$ if and only if we can never have a nontrivial subspace inclusion $Y\inclu X$ such that its dualization is an isomorphism. 
\end{itemize}
Comparing the two axioms with respect to their implications for subspaces and sublocales yields the following. First, define $\spa[\mf{S}(\Om(X))]$ as being the ordered collection of spatial sublocales of $\Om(X)$. We always have a map $\Om'\colon\ca{P}(X)\ra \spa[\mf{S}(\Om(X))]$, which to each subspace of $X$ assigns the sublocale that it induces on $\Om(X)$. The following is well-known:
\begin{itemize}
    \item A space $X$ is sober if and only the map $\Om'\colon\ca{P}(X)\ra \spa[\mf{S}(\Om(X))]$ is a surjection.
    \item A space $X$ is $T_D$ if and only if the map $\Om'\colon\ca{P}(X)\ra \spa[\mf{S}(\Om(X))]$ is an injection.
\end{itemize}

A frame $L$ is said to be  $T_D$-\emph{spatial} if  $L\cong\Omega(X)$ for some $T_D$-topological space $X$.  A \emph{covered prime} of $L$ is a $p\ne 1$ such that whenever $p=\bigwedge_i x_i$ for some $\{x_i\}_{i\in I}\subseteq L$, then there is an $i\in I$ with $p=x_{i}$. In general, it is not true that localic maps map covered primes into covered primes. Following  \cite{banaschewskitd} a frame homomorphism $f$ will be said to be a \emph{$D$-homomorphism} if its associated localic map $f_*$ sends covered primes into covered primes. The subset  $\pt_D(L)$ will denote the set of covered primes  of $L$, and we will refer to it as the \emph{$T_D$ spectrum} of $L$. For each $a\in L$, set $\Sigma'_a=\set{p\in \pt_D(L)\mid a\not\leq p}$. Then, the family $\set{ \Sigma'_a\mid a\in L}$ is a topology on $\pt_D(L)$, and in fact $(\pt_D(L),\set{\Sigma'_a\mid a\in L})$ is a \donotbreakdash{$T_D$}space  \cite{banaschewskitd}.

Let $\cat{Top}_D$ be the full subcategory of $\cat{Top}$ consisting of \donotbreakdash{$T_D$}spaces and let $\cat{Frm}_D$ be the non-full subcategory of $\cat{Frm}$ whose  morphisms are the $D$-homomorphisms.
In  \cite{banaschewskitd}, the authors showed that the assignment $\Sigma'(L)=(\pt_D(L),\set{\Sigma'_a\mid a\in L})$ yields an adjunction 

\[
\begin{tikzcd}[row sep=huge, column sep=huge, text height=1.5ex, text depth=0.25ex]
 \cat{Top}_D \arrow[bend left=20]{r}[name=a]{\cat{\Omega}} &\cat{Frm}_{D}^{op}  \arrow[bend left=20]{l}[name=b]{\Sigma'} \arrow[phantom, from=a, to=b]{}{\ \bot}
\end{tikzcd}
\]

Analogously to the case of the classical spectrum, a frame $L$ is \donotbreakdash{$T_D$}spatial iff $L\cong \Omega (\pt_D(L))$, and there is always a surjective frame homomorphism (the counit of the adjunction) $L\twoheadrightarrow \Omega(\pt_D(L))$ sending $a$ to $\Sigma'_a$. We will refer to $\Omega(\pt_D(L))$ as the \emph{\donotbreakdash{$T_D$}spatialization} of $L$.

\section{$T_D$-spatiality and strong $T_D$-spatiality}
Among the matters that we will tackle in this paper there is that of studying the system of all \donotbreakdash{$T_D$}spatial sublocales of a frame. This is why we begin with an analysis of \donotbreakdash{$T_D$}spatiality, and of a natural strengthening of this property which we will call ``strong \donotbreakdash{$T_D$}spatiality". We recall that a frame $L$ is \donotbreakdash{$T_D$}spatial if it is isomorphic to the frame of opens of the subspace of its prime spectrum consisting of the covered primes. A natural strengthening of this condition is the following. We say that a frame $L$ is \emph{strongly \donotbreakdash{$T_D$}spatial} if it is spatial and all its primes are covered. Thus, a strongly \donotbreakdash{$T_D$}spatial is isomorphic to the frame of opens of its $T_D$ spectrum, and furthermore its $T_D$ spectrum coincides with the classical (sober) spectrum. In this section we will characterize both these conditions on a frame; in particular we will show that a frame is \donotbreakdash{$T_D$}spatial if and only if the Booleanization of its assembly is spatial, and that a frame is strongly \donotbreakdash{$T_D$}spatial if and only if the Booleanization of its assembly coincides with its spatialization. 

\begin{lemma}[{\cite[Proposition~10.2]{remainders}}]\label{a'}
For a frame $L$ and a prime $p\in \pt(L)$ we have that $\bl(p)$ is complemented in $\mf{S}(L)$ if and only if $p$ is a covered prime.
\end{lemma}

From the fact that $\mf{S}_b(L)$ is the collection of joins of sublocales complemented in $\SL$, we deduce the following corollary.

\begin{corollary}\label{a}
For a frame $L$ and a prime $p\in \pt(L)$ we have that $\bl(p)\in \mf{S}_b(L)$ if and only if $\bl(p)$ is a covered prime. Then, for a frame $L$ and for $\{p_i\}_{i\in I}$ a collection of covered primes, we have that all joins of the form $\bve_i \bl(p_i)$ are in $\mf{S}_b(L)$.
\end{corollary}

We recall that a frame $L$ is spatial if and only if every element is a meet of primes. We have the following analogous characterization of \donotbreakdash{$T_D$}spatiality.

\begin{lemma}\label{me}
A frame is \donotbreakdash{$T_D$}spatial if and only if all its elements are meets of covered primes.
\end{lemma}

\begin{proof}
For the ``only if'' part, if $L=\Omega(X)$ with $X$ a \donotbreakdash{$T_D$}space and $U$ is an open, then $U=\bwe_{x\not\in U} X-\overline{\{x\}}$, where each  $X-\overline{\{x\}}$ is covered. Conversely, it is clear that if every element is a meet of covered primes, then the \donotbreakdash{$T_D$}spatialization surjection $L\twoheadrightarrow \Omega(\pt_D(L))$ defined in the Preliminaries is also one-to-one and hence an isomorphism. Since $\pt_D(L)$ is a \donotbreakdash{$T_D$}space, then $L$ is \donotbreakdash{$T_D$}spatial.
\end{proof}

We are now ready to prove our characterization theorem for \donotbreakdash{$T_D$}spatiality.

\begin{theorem}\label{ss}
For a frame $L$, the following are equivalent:
\begin{enumerate}[\normalfont(1)]
    \item $L$ is \donotbreakdash{$T_D$}spatial;
    \item Every element of $L$ is a meet of covered primes;
    \item $\SLb$ is spatial;
    \item $\SLb\cong \ca{P}(\pt_D(L))$.
\end{enumerate}
\end{theorem}
\begin{proof}
The equivalence between (1) and (3) appears in \cite[Theorem~6.2]{Arrieta20} and the equivalence between (1) and (2) follows by Lemma~\ref{me}.
Let us now show that (3) implies (4). If $\SLb$ is spatial, as this is a Boolean algebra it must be isomorphic to the powerset of its primes. By Corollary \ref{a}, there is a bijection $\pt(\SLb)\cong \pt_D(L)$. Finally note that (4) implies (3) is trivial.\qedhere
\end{proof}

We now move on to looking at strong \donotbreakdash{$T_D$}spatiality.

\begin{lemma}\label{d}
A frame $L$ is spatial if and only if the spatialization of $\mf{S}(L)^{op}$ is a dense frame surjection.
\end{lemma}
\begin{proof}
A frame $L$ is spatial if and only if $L$ is a fixpoint of the interior $\spa\colon\SL\ra \SL$, that is, if and only if the bottom element $L\in \SL^{op}$ is a fixpoint of the spatialization nucleus $\spa\colon\SL^{op}\ra \SL^{op}$.
\end{proof}
Since $\SSS_b(L)$ is the least dense sublocale of $\SSS(L)$ then one has the following:

\begin{corollary}\label{sp}
For a frame $L$ the following are equivalent:
\begin{enumerate}[\normalfont(1)]
\item $L$ is spatial;
\item $\SSS_b(L)\subseteq \spa[\SSS(L)]$.
\end{enumerate}
\end{corollary}

We may now prove the characterization theorem for strong \donotbreakdash{$T_D$}spatiality.
 \begin{theorem}\label{ssss}
For a frame $L$, the following are equivalent:
\begin{enumerate}[\normalfont(1)]
\item $L$ is strongly \donotbreakdash{$T_D$}spatial;
\item Every element of $L$ is a meet of covered primes and every prime of $L$ is covered;
\item The frame $L$ is spatial and $\pt(L)$ is a \donotbreakdash{$T_D$}space;
\item The frame $L$ is isomorphic to $\Om(X)$ for some sober \donotbreakdash{$T_D$}space $X$;
\item The frame $L$ is spatial and we have an isomorphism $\spa[\SL]\cong \ca{P}(\pt(L))$;
\item $\mf{S}_b(L)=\spa[\mf{S}(L)]$.
\end{enumerate}
\end{theorem}

\begin{proof}
(1)$\implies$(2). This follows from Lemma~\ref{me}. \\[2mm]
(2)$\implies$(3). If every element of $L$ is a meet of primes, the frame $L$ is spatial. Furthermore, if all primes of $L$ are covered we have that $\pt(L)$ is a \donotbreakdash{$T_D$}space, as it coincides with the $T_D$ spectrum of $L$.\\[2mm]
(3)$\implies$(4). This is clear as $\pt(L)$ is always sober.\\[2mm]
(4)$\implies$(5). We recall that if a space $X$ is $T_D$ then the map $\Om'\colon\ca{P}(X)\ra \spa[\mf{S}(\Om(X))]$ is injective, while if it is sober it is surjective (see the subsection on the $T_D$ axiom in the preliminaries). Thus, if $L\cong \Om(X)$ for some $T_D$ and sober space $X$ we have an isomorphism $\ca{P}(\pt(L))\cong \spa[\mf{S}(\pt(L))]$.\\[2mm]
(5)$\implies$(6). If the frame $L$ is spatial and we have an isomorphism $\spa[\SL]\cong \ca{P}(\pt(L))$, then in particular the subcolocale $\spa[\SL]\se \SL$ is Boolean. As $L$ is spatial it is also dense (see Lemma~\ref{d}), and so it must be its Booleanization.\\[2mm]
(6)$\implies$(1). By Corollary \ref{sp}, $L$ is spatial. Let $p$ be a prime of $L$. Then  $\mathfrak{b}(p)\in \spa[\SSS(L)]\subseteq \SSS_b(L)$, and so by Corollary \ref{a} it follows that $p$ is covered.\qedhere
\end{proof}

\section{$D$-sublocales}
We now investigate the notion equivalent to that of sublocale in the \donotbreakdash{$T_D$}duality; the notion of \emph{\donotbreakdash{$D$}sublocale}. For a sublocale to count as a \donotbreakdash{$D$}sublocale, we will require that its corresponding localic inclusion is a morphism in the category $\mathsf{Loc}_D$. This constraint amounts to the condition that the localic map ought to send covered primes to covered primes. In this section we will analyze the collection $\SLd\se \SL$ of \donotbreakdash{$D$}sublocales of a frame $L$. In particular, we will show that $\SLd\se \SL$ is always a dense subcolocale. We then will explore the question of how close the assignment $\SLd(-)\colon\mf{Obj}(\mathsf{Frm})\ra \mf{Obj}(\mathsf{Frm})$ is to being functorial.
It is known that localic maps do not generally send covered primes into covered primes; those frame homomorphisms such that their right adjoint localic maps send covered primes into covered primes   were referred to as $D$-homorphisms in \cite{picadopultr2011frames}.
Now we introduce one of the central notions in this paper: a sublocale $S$ of $L$ will be said to be a $D$-\emph{sublocale} if the corresponding frame surjection $L\twoheadrightarrow S$ is a $D$-homomorphism, i.e. iff the equality $\pt_D(S)=\pt_D(L)\cap S$ (equivalently the inclusion $\pt_D(S)\subseteq \pt_D(L))$ holds.

In general, not every sublocale is a \donotbreakdash{$D$}sublocale, and neither is an intersection of two \donotbreakdash{$D$}sublocales, as the following example shows:

\begin{remark}
One of the simplest examples of an intersections of two \donotbreakdash{$D$}sublocales which is not a \donotbreakdash{$D$}sublocale seems to be the following: let $L=[0,1]$ (the unit interval with its usual total order).  A subset of a totally ordered set is a sublocale iff it is closed under meets, so the following two subsets are indeed sublocales:
$$S=\set{0,1}\cup \bigcup_{n\in \mathbb{N}} \bigg[\frac{1}{2n},\frac{1}{2n-1}\bigg)\quad\textrm{and}\quad T=\set{0,1}\cup \bigcup_{n\in \mathbb{N}} \bigg[\frac{1}{2n+1},\frac{1}{2n}\bigg).$$
Obviously, for every $x\ne 1$ in $S$, one has $x=\bigwedge_{x<t\in S} t$, and this shows that $\pt_D(S)=\varnothing$, so $S$ is a \donotbreakdash{$D$}sublocale of $L$. Similarly, $\pt_D(T)=\varnothing$ and so it is also a \donotbreakdash{$D$}sublocale of $L$. Now, observe that $S\cap T=\set{0,1}$ and then one has $\pt_D(S\cap T)=\pt(S\cap T)=\{0\}$. But $0$ is not a covered prime in $L$ as $0=\bigwedge_{t>0} t$. Hence $S\cap T$ is not a \donotbreakdash{$D$}sublocale.
\end{remark}

Let $\mf{S}_D(L)$ be the subset of  $\mf{S}(L)$ consisting of \donotbreakdash{$D$}sublocales of $L$, equipped with the inherited order. We observe that we have the following.
$$\mf{S}_D(L) = \set{ S\in \SSS(L)\mid \pt_D(S)\subseteq \pt_D(L)}.$$
In what follows, we investigate the structure of $\SSS_D(L)$:

\begin{proposition}\label{jjjj}
The system $\SSS_D(L)$ is closed under arbitrary joins in $\SSS(L)$.
\end{proposition}

\begin{proof}
Let $\{S_i\}_{i\in I}\subseteq \SSS_D(L)$, i.e. $\pt_D(S_i)\subseteq \pt_D(L)$ for all $i\in I$.  We have to check that $\bigvee_i S_i\in \SSS_D(L)$, that is, $\pt_D(\bigvee_i S_i)\subseteq \pt_D(L)$. Let then $p\in \pt_D(\bigvee_i S_i)$. Since $p\in \pt_D(\bigvee_i S_i)\subseteq \bigvee_i S_i$, there is a family $\{a_i\}_{i\in I}\subseteq L$ such that $a_i\in S_i$ for each $i\in I$ and with $p=\bigwedge_i a_i$. Since for all $i\in I$ one has $a_i\in S_i\subseteq \bigvee_i S_i$ and $p\in \pt_D(\bigvee_i S_i)$, then there is an $i_0\in I$ such that $p=a_{i_0}\in S_{i_0}$.  We claim that $p\in \pt_D(S_{i_0})$. Indeed, let $\{b_j\}_{j\in J}\subseteq S_{i_0}$ a family such that $p=\bigwedge_j b_j$. Then $b_j\in S_{i_0}\subseteq \bigvee_i S_i$ and since $p\in \pt_D(\bigvee_i S_i)$, there is a $j_0\in J$ with $p=b_{j_0}$. Then $p\in \pt_D(S_{i_0})\subseteq \pt_D(L)$, as desired. 
\end{proof}

\begin{proposition}\label{impr}
If $S\in \SSS_D(L)$ and $T\in \SSS(L)$, then $S\setminus T\in \SSS_D(L)$.
\end{proposition}

\begin{proof}
Let $p\in \pt_D(S\setminus T)$. Since $\SSS(L)^{op}$ is zero-dimensional, write $T=\bigcap_i C_i$ with $C_i$ complemented.   Then $S\setminus T=S\setminus \bigcap_i C_i= \bigvee_i S\setminus C_i =\bigvee_i S\cap C_{i}^{c}$. Now, $p\in S\setminus T=\bigvee_i S\cap C_{i}^{c}$ and so there is a family $\{a_i\}_{i\in I}$ with $a_i\in S\cap C_{i}^{c}$ and $p=\bigwedge_i a_i$. Since $a_i\in S\cap C_{i}^{c}\subseteq S\setminus T$ and $p\in \pt_D(S\setminus T)$, then there is an $i_0\in I$ with $p=a_{i_0}\in C_{i_0}^c$. Since $T=\bigcap_i C_i$, then we conclude $p\not\in T$. Now, $p\not\in T$ is equivalent to $\mathfrak{b}(p)\not\subseteq T$, which is in turn equivalent to $\mathfrak{b}(p)\setminus T\ne 0$. But $\mathfrak{b}(p)\setminus T\subseteq \mathfrak{b}(p)=\set{1,p}$, and so $\mathfrak{b}(p)\setminus T=\mathfrak{b}(p)$. 
I claim that $\mathfrak{b}(p)\not\subseteq S\setminus \mathfrak{b}(p)$. By way of contradiction assume otherwise. Then $(S\setminus T)\setminus \mathfrak{b}(p)= (S\setminus \mathfrak{b}(p))\setminus T\supseteq  \mathfrak{b}(p)\setminus T=\mathfrak{b}(p)$. Then $0\ne \mathfrak{b}(p)=\mathfrak{b}(p)\cap [(S\setminus T)\setminus \mathfrak{b}(p)]$, which is a contradiction since $\mathfrak{b}(p)$ is complemented in $S\setminus T$. Therefore, we must have $\mathfrak{b}(p)\not\subseteq S\setminus \mathfrak{b}(p)$, which is equivalent to $0=\mathfrak{b}(p)\cap(S\setminus \mathfrak{b}(p))$. Then $\mathfrak{b}(p)$ is complemented in $S$, i.e. $p\in \pt_D(S)\subseteq \pt_D(L)$, as we wanted.
\end{proof}

\begin{corollary}
$\SSS_D(L)$ is a dense $D$-subcolocale of $\SSS(L)$. In particular it is a coframe.
\end{corollary}

\begin{proof}
The fact that it is a subcolocale follows from the two previous propositions. Moreover, density  follows from the obvious fact that $L$ (i.e. the bottom element of $\SSS(L)^{op}$) belongs to $\SSS_D(L)$. Finally, since $\SSS(L)^{op}$ is a zero-dimensional frame, it is in particular regular; and it is well-known that primes in any regular frame are maximal (so in particular they are covered). Thus, every sublocale of $\SSS(L)^{op}$ is a \donotbreakdash{$D$}sublocale.
\end{proof}

Density of $\SSS_D(L)$ has an important consequence:  it implies that $\SSS_b(L)\subseteq \SSS_D(L)$. Hence, all open sublocales and closed sublocales are \donotbreakdash{$D$}sublocales, and so are locally closed sublocales or more generally smooth sublocales.

The system $\SSS_D(L)$ contains several different families of sublocales, for example:

\begin{enumerate}
\item All smooth sublocales, as we have just pointed out;

\item All pointless sublocales of $L$ (as $\pt_D(S)\subseteq \pt(S)=\varnothing$);

\item Any join of pointless sublocales (joins of pointless sublocales may contain points, but their join will still be a \donotbreakdash{$D$}sublocale because  of Proposition~\ref{jjjj});

\item All codense sublocales of $L$ in which primes are maximal (these include codense naturally Hausdorff sublocales, codense fit sublocales,\dots). This only of interest if $L$ not subfit, cf. \cite{cleme18}. In order to show this assertion, one easily checks that if $p$ is a maximal element in a codense sublocale $S$, then it is maximal in $L$, and in particular it belongs to $\pt_D(L)$.
\end{enumerate}

\subsection{Functoriality}
As we will see below, the assignment $L\mapsto \SSS_D(L)^{op}$ cannot be made functorial in the whole of the category of frames in such a way that there is a natural transformation $\cl\colon 1_{\cat{Frm}}\to \SSS_D(-)^{op}$.  Therefore, one has to deal with \emph{lifts} of individual frame homomorphisms. Let   $f\colon L\to M$ be a frame homomorphism and suppose there is a further frame homomorphism $\SSS_D(f)\colon \SSS_D(L)^{op}\to \SSS_D(M)^{op}$ together with a commutative square
\[\begin{tikzcd}
\SSS_D(L)^{op} \arrow{r}{\SSS_D(f)}  &\SSS_D(M)^{op}   \\
L \arrow{r}{f}  \arrow[rightarrowtail]{u}{\cl_L} & M \arrow[rightarrowtail]{u}{\cl_M}
\end{tikzcd}\]
If $S$ is any sublocale of $L$, then one can write $S=\bigcap_{S\subseteq \op(a)\vee\cl(b)}\op(a)\vee\cl(b)$. Each of the sublocales $\op(a)\vee\cl(b)$ belongs to $\SSS_D(L)$ (as open sublocales and closed sublocales are \donotbreakdash{$D$}sublocales and $\SSS_D(L)$ is closed under joins). Now, if $S$ happens to be a \donotbreakdash{$D$}sublocale, then the intersection $\bigcap_{S\subseteq \op(a)\vee\cl(b)}\op(a)\vee\cl(b)$ coincides with the meet $\bigwedge_{S\subseteq \op(a)\vee\cl(b)}\op(a)\vee\cl(b)$ taken in $\SSS_D(L)$ (incidentally, this, together with the fact that $\op(a)\vee\cl(b)$ are complemented in $\SSS_D(L)$, shows that $\SSS_D(L)^{op}$ is a zero-dimensional frame). Hence, each $S\in\SSS_D(L)$ can be written as $S=\bigvee_{ \op(a)\wedge\cl(b)\leq S}\op(a)\wedge\cl(b)$ (here we are using the reverse order), where $\bigvee$ is nothing but the join in $\SSS(L)^{op}$.  Now, applying the frame homomorphism  $\SSS_D(f)$ and using the fact that frame homomorphisms commute with complements, we obtain
$\SSS_D(f)(S)=\bigvee_{\op(a)\wedge\cl(b)\leq S} \cl(f(a))\wedge \op(f(a))$, where  $\bigvee$ denotes join in $\SSS_D(M)^{op}$. Denote $g=f_*$. Then, one has that $\SSS_D(f)(S)$ can be computed as the largest \donotbreakdash{$D$}sublocale of $M$ contained in $g_{-1}[S]$. 

\begin{proposition}
Let $L$ a frame and $f\colon L\twoheadrightarrow S$ a surjection onto a sublocale $S$. Then $f$ lifts if and only if it is  a $D$-homomorphism (i.e. iff $S$ is a \donotbreakdash{$D$}sublocale of $L$).
\end{proposition}

\begin{proof}
For the ``if'' implication, assume that $S$ is a \donotbreakdash{$D$}sublocale of $L$ and define a map $h\colon \SSS_D(L)\to \SSS_D(S)$ given by $h(T)=T\wedge S$ (where $\wedge$ stands for meet in $\SSS_D(L)$). We note that $h$ is well-defined because since $T\wedge S\in \SSS_D(L)$ in particular we have $T\wedge S\in \SSS_D(S)$. Also, since $S$ is a \donotbreakdash{$D$}sublocale of $L$, it is easy to check that $h$ preserves arbitrary meets (and obviously finite joins as well). Hence it is a (surjective) coframe homomorphism. For showing that the relevant naturality square commutes it suffices to check that $S\cap \cl(a)$ is a \donotbreakdash{$D$}sublocale of $L$. Let $p\in\pt_D(S\cap \cl(a))$. If $p=\bigwedge_i x_i$ for a family $\{x_i\}_{i\in I}\subseteq S$, then since $p\in \cl(a)$, one has $a\leq p\leq x_i$ for each $i\in I$ and so $x_i\in S\cap \cl(a)$ for each $i\in I$. Hence there is an $i_0\in I$ with $p=x_{i_0}$. Thus $p\in \pt_D(S)\subseteq \pt_D(L)$.

Let us now show the converse, so assume that there is a frame homomorphism \linebreak $h\colon \SSS_D(L)^{op}\to \SSS_D(S)^{op}$ making the square
\[\begin{tikzcd}
\SSS_D(L)^{op} \arrow{r}{h}  &\SSS_D(S)^{op}   \\
L \arrow[twoheadrightarrow]{r}{f}  \arrow[rightarrowtail]{u}{\cl_L} & S \arrow[rightarrowtail]{u}{\cl_S}
\end{tikzcd}\]
commutative. We claim that $h$ is a surjection too. Indeed, we noted in the discussion preceding the statement that  $\SSS_D(S)^{op}$ is generated by $\set{\cl_S(a),\op_S(b)\mid a,b\in S}$, and thus it suffices to observe that $h(\cl_L(a))=\cl_S(a)$ and $h(\op_L(a))=\op_S(a)$ for each $a,b\in S$. Hence $h$ is a surjection. Let us finally check that $S$ is a \donotbreakdash{$D$}sublocale of $L$. Let $p\in\pt_D(S)$.  Then by the converse implication, the surjection $S\twoheadrightarrow \mathfrak{b}(p)$ lifts, and so we obtain a commutative diagram as follows:
\[\begin{tikzcd}
\SSS_D(L)^{op} \arrow[twoheadrightarrow]{r}{h}  &\SSS_D(S)^{op}  \arrow[twoheadrightarrow]{r}{} & \SSS_D(\mathfrak{b}(p))^{op}  \\
L \arrow[twoheadrightarrow]{r}{f}  \arrow[rightarrowtail]{u}{\cl_L} & S \arrow[rightarrowtail]{u}{\cl_S}  \arrow[twoheadrightarrow]{r}{} & \mathfrak{b}(p) \arrow{u}{\cong}
\end{tikzcd}\]
where the top composite is also a surjection. Hence there is a surjection $\SSS_D(L)^{op}\twoheadrightarrow \mathfrak{b}(p)$ which displays $\mathfrak{b}(p)$ as a point of $\SSS_D(L)^{op}$, but Lemma~\ref{covprimofSLd} below asserts than then $p$ must be a covered prime of $L$.
\end{proof}

\begin{remark}Characterizing   monomorphisms which lift seems to be a more difficult task. An analogous situation happens when studying the functoriality of the assignment $L\mapsto \SSS_b(L)$ (cf. \cite[Proposition~5.10]{Arrieta20}).\end{remark}

\section{The relation between $D$-sublocales and subspaces}

In \cite{Suarez20} the matter of the relation between subspaces and sublocales is explored. Recall that there it is proven that we have an adjunction
\[
\begin{tikzcd}[row sep=huge, column sep=huge, text height=1.5ex, text depth=0.25ex]
 \mathcal{P}(\pt(L))  \arrow[bend left=20]{r}[name=a]{\mathfrak{M}} & \mf{S}(L) \arrow[bend left=20]{l}[name=b]{\pt} \arrow[phantom, from=a, to=b]{}{\ \bot}
\end{tikzcd}
\]
in which the fixpoints are precisely the spatial sublocales of $\SL$ on one side, and the sober subspaces of $\pt(L)$ on the other. It is also shown that for every frame $L$ we have a commuting diagram, as follows, in the category of coframes.

\begin{center}
\begin{tikzcd}[row sep=large, column sep = large]
\spa [\mf{S}(L)]  
\arrow{r}{\pt(\cong)}  
& \sob[\ca{P}(\pt(L))] 
\arrow[hookrightarrow]{d}  \\
\mf{S}(L) 
\arrow{r}{\pt} 
\arrow[twoheadrightarrow]{u}{\spa} 
& \ca{P}(\pt(L))  
\end{tikzcd}
\end{center}

The top horizontal arrow is always an isomorphism. It is them meaningful to ask three questions. 
\begin{enumerate}
    \item Firstly we can ask ourselves when is it that the left vertical arrow is an isomorphism. The frames for which this holds are the \tc{totally spatial} frames, i.e. those such that all their sublocales are spatial. These are precisely the frames $L$ such that $\SL^{op}$ is spatial, and they have also been characterized in several other ways in \cite{niefield87}. By looking at the diagram, we see that the totally spatial frames are exactly those such that there is a perfect correspondence between the sublocales of $L$ and the sober subspaces of $\pt(L)$.
    \item We may also ask when it is that the right vertical arrow is an isomorphism. The frames for which this holds are exactly those such that $\pt(L)$ is a \donotbreakdash{$T_D$}space, as shown in \cite{Suarez20}. By looking at the diagram, we see that these are exactly the frames for which there is a perfect correspondence between the spatial sublocales of $L$ and the subspaces of $\pt(L)$.
    \item We may ask ourselves when it is that both vertical arrows are isomorphisms. These frames are the spatial frames such that $\pt(L)$ is a \tc{scattered} space (see \cite{simmons80}). These are the frames such that there is a perfect correspondence between sublocales of $L$ and subspaces of $\pt(L)$.
\end{enumerate}

We want to tackle the analogues of these three questions in the setting of the \donotbreakdash{$T_D$}duality. Let us start from the main adjunction.

\begin{proposition}\label{mainadjunction}
There is an adjunction 
\[
\begin{tikzcd}[row sep=huge, column sep=huge, text height=1.5ex, text depth=0.25ex]
 \mathcal{P}(\pt_D(L))  \arrow[bend left=20]{r}[name=a]{\mathfrak{M}} & \mf{S}_D(L) \arrow[bend left=20]{l}[name=b]{\pt_D} \arrow[phantom, from=a, to=b]{}{\ \bot}
\end{tikzcd}
\]
The fixpoints of $\mathfrak{M}\circ \pt_D$ are the \donotbreakdash{$T_D$}spatial \donotbreakdash{$D$}sublocales, and $\pt_D\circ \mathfrak{M}$ is the identity.
\end{proposition}
\begin{proof}
The map $\pt_D$ is well-defined because if $S\in \SSS_D(L)$ then $\pt_D(S)\subseteq \pt_D(L)$. Moreover, clearly both $\pt_D$ and $\mathfrak{M}$ are monotone. For the adjunction, we have to check that $\mathfrak{M}(Y)\subseteq S$ if and only if $Y\subseteq \pt_D(S)$ for each $S\in \SSS_D(L)$ and $Y\subseteq \pt_D(L)$. The ``only if'' implication is trivial because $Y\subseteq \mathfrak{M}(Y)$, and the converse clearly follows from the fact that sublocales are closed under meets.
Now, By Lemma~\ref{me}, the fixpoints of $\mathfrak{M}\circ \pt_D$ are precisely those \donotbreakdash{$D$}sublocales which are \donotbreakdash{$T_D$}spatial. The only task remaining is to show that $\pt_D\circ \mathfrak{M}$ is the identity, so let $Y\subseteq \pt_D(L)$ and $p\in \pt_D(\mathfrak{M}(Y))$. Since $p\in \pt_D(\mathfrak{M}(Y))\subseteq \mathfrak{M}(Y)$, there is a family $\{y_i\}_{i\in I}\subseteq Y$ with $p=\bigwedge y_{i_0}$. But $p$ is a covered prime in ${Y}$ so there is an $i_0\in I$ with $p=y_{i_0}\in Y$. The other inclusion follows by adjunction.
\end{proof}

\begin{lemma}\label{tdsub}
For a frame $L$, the sublocale of $L$ associated with its \donotbreakdash{$T_D$}spatialization is $\mathfrak{M}(\pt_D(L))$.
\end{lemma}
\begin{proof}
By the universal property of the \donotbreakdash{$T_D$}spatialization, this is the largest \donotbreakdash{$D$}sublocale of $L$ such that it is \donotbreakdash{$T_D$}spatial. Suppose that $S\se L$ is a \donotbreakdash{$T_D$}spatial \donotbreakdash{$D$}sublocale. Because $S$ is a \donotbreakdash{$D$}sublocale, we have $\pt_D(S)\se \pt_D(L)$, and because $\mathfrak{M}$ is monotone we also have $\mathfrak{M}(\pt_D(S))\se \mathfrak{M}(\pt_D(L))$. Thus, by \donotbreakdash{$T_D$}spatiality of $S$, we have proven that $S\se \mathfrak{M}(\pt_D(L))$.
\end{proof}

\subsection{Local and global $T_D$-spatialization}

It follows from Lemma~\ref{tdsub} that the map $\mathfrak{M}\circ \pt_D\colon\mf{S}_D(L)\ra \mf{S}_D(L)$ assigns to each \donotbreakdash{$D$}sublocale its \donotbreakdash{$T_D$}spatialization. From now on, we will call this map $\spa_D\colon\SLd\ra \SLd$.
 Let us see what more we can learn about this operator. 

\begin{lemma}\label{interior1}
If $L$ is a complete lattice and $\iota\colon L\ra L$ is an interior operators, then $\iota[L]$ is a complete lattice where the joins in are computed as in $L$, and the meets are computed as $\bwe^{\iota}\iota(x_i)=\iota(\bwe _i \iota(x_i))$.
\end{lemma}
\begin{proof}
Suppose that $L$ is a complete lattice and that $\iota\colon L\ra L$ is an interior operator. Suppose that $x_i\in L$. We claim that $\bve_i \iota(x_i)$ is the join of $\set{\iota(x_i)\mid i\in I}$ in $\iota[L]$. first, we show that $\iota(\bve_i \iota(x_i))=\bve_i \iota(x_i)$. One inequality holds since $\iota$ is deflationary. For the other, we notice that $\iota(x_i)\leq \bve_i \iota(x_i)$. By monotonicity and idempotence of $\iota$, it follows that $\iota(x_i)\leq \iota(\bve_i \iota(x_i))$. The desired result follows. Let us show that $\bve_i^{\iota}\iota(x_i)=\iota(\bve_i \iota(x_i))$. On the one hand we showed before that $\iota(x_i)\leq \iota(\bve_i \iota(x_i))$. On the other hand, if we have $\iota(x_i)\leq \iota(y)$ for some $y\in L$ we also have that $\iota(\bve_i\iota(x_i))\leq \iota(y)$, by monotonicity and idempotence of $\iota$. Then, indeed $\bve_i^{\iota}\iota(x_i)=\iota(\bve_i \iota(x_i))$. Now let us show that $\bwe_i^{\iota}\iota(x_i)=\iota(\bwe_i \iota(x_i))$. On the one hand we have $\bwe_i \iota(x_i)\leq \iota(x_i)$ and so $\iota(\bwe_i \iota(x_i))\leq \iota(x_i)$ by monotonicity and idempotence. On the other hand if $\iota(y)\leq \iota(x_i)$ then we also have $\iota(y)\leq \bwe_i \iota(x_i)$ and so $\iota(y)\leq \iota(\bwe_i \iota(x_i))$ by monotonicity and idempotence.
\end{proof}

\begin{lemma}\label{interior2}
If $L$ is a complete lattice and $\iota\colon L\ra L$ an interior operator on it, the surjection $\iota\colon L\epi \iota[L]$ preserves arbitrary meets. If $\iota\colon L\ra L$ preserves finite joins, this surjection preserves finite joins, too.
\end{lemma}
\begin{proof}
Suppose that $\iota\colon L\ra L$ is an interior operator. Suppose that $x_i\in L$. We first claim that $\iota(\bwe_i x_i)=\iota(\bwe_i \iota(x_i))$. We have that $\iota(x_i)\leq x_i$ since $\iota$ is deflationary, and thus $\bwe _i \iota(x_i)\leq \bwe_i x_i$, from which we deduce $\iota(\bwe_i\iota(x_i))\leq \iota(\bwe_i x_i)$ by monotonicity of $\iota$. For the reverse inequality, we start from the inequality $\bwe_i x_i\leq x_i$. By monotonicity of $\iota$ we deduce $\iota(\bwe_i x_i)\leq \bwe_i \iota(x_i)$. Finally, by monotonicity and idempotence, we deduce $\iota(\bwe_i x_i)\leq \iota(\bwe_i \iota(x_i))$. We have shown that $\iota(\bwe_i x_i)=\iota(\bwe_i \iota(x_i))=\bwe^{\iota}\iota(x_i)$, where the last equality holds by Lemma \ref{interior1}. Suppose now that $\iota\colon L\ra L$ preserves finite joins. For $x,y\in L$ we have that $\iota(x\ve y)=\iota(x)\ve \iota(y)=\iota(x)\ve^{\iota}\iota(y)$, where the last equality holds by Lemma~\ref{interior1}. 
\end{proof}

\begin{proposition}\label{spaDint}
For a frame $L$, the map $\spa_D\colon \SLd\ra \SLd$ is an interior operator which preserves finite joins.
\end{proposition}
\begin{proof}
The map is monotone since it is the composition of the monotone maps $\pt_D\colon \SLd\ra \ca{P}(\pt_D(L))$ and $\mathfrak{M}\colon \ca{P}(\pt_D(L))\ra \SLd$. The map is clearly deflationary. Idempotence follows from the adjunction identity $\mathfrak{M}\circ \pt_D\circ \mathfrak{M}=\mathfrak{M}$. For preservation of finite joins, suppose that we have a finite collection $S_m\in \SLd$. Since these are \donotbreakdash{$D$}sublocales and by primality of the elements in $\pt_D(L)$, we have that $\pt_D(S_1\ve ...\ve S_n)=\pt_D(L)\cap (S_1\ve ...\ve S_n)=\pt_D(L)\cap (S_1\cup ...\cup S_n)=(\pt_D(L)\cap S_1)\cup ...\cup (\pt_D(L)\cap S_n)=\pt_D(S_1)\cup ...\cup \pt_D(S_n)$. We have shown that $\pt_D\colon\SLd\ra \ca{P}(\pt_D(L))$ preserves finite joins. As $\mathfrak{M}$ is a left adjoint, it preserves all joins, in particular finite ones. Thus the composition $\mathfrak{M}\circ \pt_D\colon\SLd\ra \SLd$ preserves finite joins. 
\end{proof}

\begin{corollary}
Suppose that $L$ is a frame.cThe map $\spa_D\colon\SLd\epi \spa_D[\SLd]$ is a coframe surjection whose codomain is the ordered collection of the \donotbreakdash{$T_D$}spatial \donotbreakdash{$D$}sublocales of $L$.
\end{corollary}
\begin{proof}
The fact that we have a coframe surjection follows from Lemma~\ref{interior2} and from Proposition~\ref{spaDint}. The fact that the codomain is the ordered collection of the \donotbreakdash{$T_D$}spatial \donotbreakdash{$D$}sublocales follows from Corollary \ref{mainadjunction}.
\end{proof}

\begin{remark}
The operator $\spa_D\colon\SLd\ra \SLd$ actually preserves arbitrary joins, too. In the proof of Lemma~\ref{spaDint}, we have shown that for a finite collection $S_m$ of sublocales we have that $\pt_D(S_1\ve ...\ve S_n)=\pt_D(S_1)\cup...\cup p\pt_D(S_n)$. Using coveredness of primes, this argument can be strengthened to show that $\pt_D(\bve_i S_i)=\bcu_i \pt_D(S_i)$ for an arbitrary collection $S_i$ of sublocales. This tells us that in the frame $\SLd^{op}$ the congruence associated with the sublocale $\spa_D[\SLd]$ is a \emph{complete} congruence (see \cite{cleme18}).
\end{remark}

We then have a local \donotbreakdash{$T_D$}spatialization surjection $\spa_D\colon\SLd\epi \spa_D[\SLd]$. As $\SLd^{op}$ is a frame, it has its own \donotbreakdash{$T_D$}spatialization sublocale; its global \donotbreakdash{$T_D$}spatialization subcolocale, so to speak. We want to know how the local \donotbreakdash{$T_D$}spatialization subcolocale $\spa[\SLd]\se \SLd$ relates to the global one. 

\begin{lemma}\label{joinofbp}
For a frame $L$, a \donotbreakdash{$D$}sublocale $S\se L$ is \donotbreakdash{$T_D$}spatial if and only if $S=\bve \set{\bl(p)\mid p\in \pt_D(S)}$.
\end{lemma}
\begin{proof}
By Corollary \ref{mainadjunction}, a \donotbreakdash{$D$}sublocale $S\se L$ is \donotbreakdash{$T_D$}spatial if and only if $\mathfrak{M}(\pt_D(S))=S$. But $\mathfrak{M}(\pt_D(S))=\bve \set{\bl(p)\mid p\in \pt_D(S)}$ and so we are done.
\end{proof}

\begin{lemma}\label{covprimofSLd}
The covered prime elements of $\SLd^{op}$ are exactly the sublocales of the form $\bl(p)$ with $p\in \pt_D(L)$.
\end{lemma}
\begin{proof}
Suppose that $p\in \pt_D(L)$. We claim that $\bl(p)\in \SLd^{op}$ is a covered prime. Indeed, we observe that if we have $\bl(p)= \bve_i S_i$ we clearly must have $S_j=\set{1,p}$ for some $j\in I$. 

For the converse, suppose that $S\in \SLd$ and that $S\neq \bl(p)$ for each $p\in \pt_D(L)$. Since for each prime $p\in \pt(L)$ we have that $\bl(p)\in \SLd$ if and only if $p$ is covered, we must have that $S$ cannot be a two-element sublocale, and so it contains at least three elements. Suppose, then, that $\set{x,y,1}$ are pairwise distinct and that $\set{x,y}\se S$. Suppose, in particular, that $x\nleq y$. First, recall that all open and all closed sublocales of $L$ are elements of $\SLd$. We claim that $S\se \op(x)\ve \cl(x)$ is a nontrivial decomposition of $S$. We have that $x\ra x=1\neq x$, and so $x\notin \op (x)$, hence $S\nsubseteq \op (x)$. On the other hand, we have that $y\notin \cl(x)$ and so $S\nsubseteq \cl(x)$. But since $\cl(x)\ve \op (x)=1$ we have $S\se \op(x)\ve \cl(x)$.
\end{proof}

\begin{observation}\label{obs1}
In the proof above, starting from the assumption that a sublocale $S\in \SLd$ is not of the form $\bl(p)$ for a covered prime $p\in L$ we have shown that $S$ is not a prime. This means that in $\SLd$ if an element is not a covered prime it is not a prime. 
\end{observation}

We have then proved the following.
\begin{proposition}\label{SLdspatial}
In the frame $\SLd^{op}$ all primes are covered. So, the frame $\SLd^{op}$ is \donotbreakdash{$T_D$}spatial if and only if it is spatial.
\end{proposition}
\begin{proof}
This follows from Proposition~\ref{covprimofSLd} and from Observation~\ref{obs1}.
\end{proof}

We are ready to show that the local and the global spatialization of $\SLd$ coincide.
\begin{proposition}\label{locandglob}
For any frame $L$ we have that $\mathfrak{M}(\pt_D(\SLd))=\spa_D[\SLd]$. In other words, global and local \donotbreakdash{$T_D$}spatialization of $\SLd$ coincide.
\end{proposition}
\begin{proof}
By Lemma~\ref{covprimofSLd}, we have that $\pt_D(\SLd)=\set{\bl(p)\mid p\in \pt_D(L)}$. Hence the elements of $\mathfrak{M}(\pt_D(\SLd))$ are exactly the joins of elements of the form $\bl(p)$ for $p\in \pt_D(L)$. By Lemma~\ref{joinofbp}, these are exactly the \donotbreakdash{$T_D$}spatial \donotbreakdash{$D$}sublocales of $L$.
\end{proof}

\begin{corollary}
For a frame $L$, we have that all its \donotbreakdash{$D$}sublocales are \donotbreakdash{$T_D$}spatial if and only if $\SLd$ is \donotbreakdash{$T_D$}spatial, and this holds if and only if $\SLd$ is spatial.
\end{corollary}
\begin{proof}
This follows from Propositions~\ref{locandglob} and \ref{SLdspatial}. 
\end{proof}

\begin{remark}
It should be noted that $\spa_D[\SSS_D(L)]$ is not only spatial, but it is also Boolean for any frame $L$. Indeed, it follows by virtue of the adjunction in Proposition~\ref{mainadjunction} and the fact that $\pt_D\circ \mathfrak{M}$ is the identity that it is actually isomorphic to $\mathcal{P}(\pt_D(L))$. More precisely, it corresponds to the Boolean sublocale of $\SSS(L)^{op}$ determined by the element $\spa_D(L)$, i.e. we have $\spa_D[\SSS_D(L)]=\mathfrak{b}(\spa_D(L))$ (to see this, observe that a Boolean sublocale $S$ of a frame $L$ always equals $\mathfrak{b}(\bigwedge S)$, hence $\spa_D[\SSS_D(L)]$ is determined by the largest \donotbreakdash{$D$}sublocale which is \donotbreakdash{$T_D$}spatial, i.e. the \donotbreakdash{$T_D$}spatialization of $L$).
\end{remark}

We wish to stress that the behaviour of the operator $\pt_D\circ \mathfrak{M}\colon \ca{P}(\pt_D(L))\ra \ca{P}(\pt_D(L))$ exhibits a difference between the spatial-sober duality and the \donotbreakdash{$T_D$}duality. In the classical duality (see \cite{Suarez20}) we have an adjunction
\[
\begin{tikzcd}[row sep=huge, column sep=huge, text height=1.5ex, text depth=0.25ex]
 \mathcal{P}(\pt(L))  \arrow[bend left=20]{r}[name=a]{\mathfrak{M}} & \mf{S}(L) \arrow[bend left=20]{l}[name=b]{\pt} \arrow[phantom, from=a, to=b]{}{\ \bot}
\end{tikzcd}
\]

in which the fixpoints of $\mathfrak{M}\circ \pt$ are precisely the spatial sublocales, and the fixpoints of $\pt\circ \mathfrak{M}$ are the sober subspaces of $\pt(L)$. None of the two compositions is the identity. In particular, since not all subspaces of $\pt(L)$ are sober, this means that the map $\pt\colon\SL\ra \ca{P}(\pt(L))$ will not in general be a surjection: some subspaces of $\pt(L)$ do not arise as spectra of sublocales of $L$. This accounts for part of the mismatch between subspaces and sublocales. In the \donotbreakdash{$T_D$}duality, this does not happen. Since the operator $\pt_D\circ \mathfrak{M}$ is the identity, it means that all subspaces of $\pt_D(L)$ arise as spectra of \donotbreakdash{$D$}sublocales of $L$. We shall see that this means that in the \donotbreakdash{$T_D$}duality the mismatch between subspaces and sublocales is only given by the map $\pt_D\colon\SLd\ra \ca{P}(\pt_D(L))$ not being injective: in the \donotbreakdash{$T_D$}duality we have more sublocales than subspaces. 

Let us now look at the $T_D$ analogue of the square in the beginning of this section.

\bigskip
\[
\begin{tikzcd}[row sep=large, column sep=large]
(\mathfrak{M}\circ \pt_D)[\mf{S}_D(L)]
\arrow{r}{\pt_D(\cong)}
& (\pt_D\circ \mathfrak{M})[\ca{P}(\pt_D(L))]
\arrow[d,"="]\\
\mf{S}_D(L)
\arrow{r}{\pt_D}
\arrow[u,twoheadrightarrow,"\mathfrak{M}\circ \pt_D"]
& \ca{P}(\pt_D(L))
\end{tikzcd}\]

\bigskip

We know that the diagram commutes in the category of coframes as the composition $\pt_D\circ \mathfrak{M}$ is the identity and so indeed $\pt_D\circ \mathfrak{M}\circ \pt_D=\pt_D$. We observe that the right vertical arrow is always the identity, by Corollary \ref{mainadjunction}. Thus, in the \donotbreakdash{$T_D$}duality, we are left with only one of the three questions. We shall then characterize those frames $L$ such that all the \donotbreakdash{$D$}sublocales of $L$ are \donotbreakdash{$T_D$}spatial, and these coincide with the frames $L$ such that there is a perfect correspondence between \donotbreakdash{$D$}sublocales of $L$ and subspaces of $\pt_D(L)$.

\subsection{When do the $D$-sublocales perfectly represent the subspaces?}

For a frame $L$ and a meet $\bwe M$, we say that $m\in M$ is an \tc{essential} element of the meet $\bwe M$ if we have $\bwe M{\sm}\{m\}\neq \bwe M$. If $L$ is a frame and  $a\in L$ is the meet of the primes above it, we say that $p\in \pt(\up a)$ is an \tc{essential} prime of $a$ if and only if $\bwe \pt(\up a){\sm}\up p\neq a$. We say that $p\in \pt(\up a)$ is an \tc{absolutely essential} prime of $a$ if $\bwe \pt(\up a){\sm}\{p\}\neq a$. It is clear that absolute essentiality is a condition stronger than essentiality. The terminology ``absolutely" essential is justified by the following result.

\begin{lemma}\label{essentialprimes}
For a frame $L$ and for $a\in L$ with $a=\bwe \pt(\up a)$ and for $p\in \pt(\up a)$, the following are equivalent.
\begin{enumerate}[\normalfont(1)]
    \item The prime $p$ is an absolutely essential prime of $a$;
    \item Whenever $a=\bwe P$ for $P\se \pt(L)$ we must have $p\in P$;
    \item The prime $p$ is weakly covered and it is an essential prime of $a$.
\end{enumerate}
\end{lemma}
\begin{proof}
Let $L$ be a frame, and let $a\in L$ be an element such that it is the meet of the primes above it. Let $p\in \pt(\up a)$.

(1)$\implies$(2). Suppose that $p$ is an absolutely essential prime of $a$, and that $P\se \pt(L)$ is such that $a=\bwe P$. We must have that $P\se \pt(\up a)$ and so $P{\sm}\{p\}\se \pt(\up a){\sm}\{p\}$, from which we deduce that $\bwe \pt(\up a){\sm}\{p\}\leq \bwe P{\sm}\{p\}$. From this fact, and by absolute essentiality of $p$, we deduce that $\bwe P{\sm}\{p\}\nleq a$. As $\bwe P\leq a$, we must have $p\in P$.

(2)$\implies$(3). Suppose that $p\in P$ whenever $\bwe P=a$ for $P\se \pt(L)$. In particular, we have that $\bwe \pt(\up a){\sm}\up p\nleq a$ as $p\notin \pt(\up a){\sm}\up p$. Now suppose towards contradiction that $p$ is not weakly covered, and so that $p=\bwe \pt(\up p){\sm}\{p\}$. We then have that $\bwe \set{q\in \pt(\up a)\mid p\neq q}=\bwe \set{q\in \pt(\up a)\mid p<q}\we \bwe \set{q\in \pt(\up a)\mid p\nleq q}=p\we \bwe \pt(\up a){\sm}\up p\leq a$. We have then contradicted (2).

(3)$\implies$(1). Suppose that $p$ is an essential prime of $a$ and that it is weakly covered. if we had $\bwe \pt(\up a){\sm}\up p\leq p$ then we would have that $\bwe \pt(\up a){\sm}\up p\leq q$ whenever $q\in \pt(\up a)$, contradicting essentiality of $p$. Then, we must have $\bwe \pt(\up a){\sm}\up p\nleq p$. We also have $\bwe\set{q\in \pt(\up a)\mid p<q}\nleq p$ by weak coveredness of $p$. By primality of $p$, this means that $\bwe \pt(\up a){\sm}\up p\we \bwe\set{q\in \pt(\up a)\mid p<q}=\bwe \pt(\up a){\sm}\{p\}\nleq p$, and so $\bwe \pt(\up a){\sm}\{p\}\nleq a$. Thus, $p$ is an absolutely essential prime of $a$.
\end{proof}

Thus, we have called this condition ``absolute" essentiality as opposed to relative essentiality: an essential prime of $a$ is only essential relative to the particular meet of primes $\bwe \pt(\up a){\sm}\set{q\in \pt(\up a)\mid p<q}$; whereas an absolutely essential prime is essential with respect to any meet of primes $\bwe P$ with $\bwe P=a$. It is crucial that for covered primes these two conditions collapse together.

\begin{lemma}\label{coveredess}
If $p$ is a covered prime and $p$ is an essential prime of $a\in L$, then $p$ is also an absolutely essential prime of $a$.
\end{lemma}
\begin{proof}
If $p$ is an essential prime of $a$ and it is covered, in particular it is weakly covered and so it is absolutely essential, by Lemma~\ref{essentialprimes}.
\end{proof}

In \cite{Suarez20}, the following is shown. This is a choice-free adaptation of a famous result of \cite{niefield87}.
\begin{proposition}\label{totspaess}
For a frame $L$, we have that $L$ is totally spatial if and only if every element of $L$ is the meet of its essential primes.
\end{proposition}

Let us prove some general facts about essential and absolutely essential primes.
\begin{lemma}\label{booleans}
For a frame $L$ and an element $a$ which is the meet of the primes above it, we have that $p\in \pt(\up a)$ is an essential prime of $a$ if and only if it is a prime of $\bl(a)$.
\end{lemma}
\begin{proof}
Suppose that $L$ is a frame and that $a\in L$ is the meet of the primes above it. Suppose first that $p$ is an essential prime of $a$. We claim that $p\in \bl(a)$ because we have
\[
p=\bwe \pt(\up a){\sm}\up p\ra a=\bwe \set{\bwe \pt(\up a){\sm}\up p\ra q\mid q\in \pt(\up a)}.
\]
The right hand side is a conjunction, in which every conjunct is either $1$ or a prime above $a$. Then, this is equal to
\[
\bwe \set{q\in \pt(\up a)\mid \bwe \pt(\up a){\up}p\nleq q}
\]
As $p$ is an essential prime, it is certainly in the set $\set{q\in \pt(\up a)\mid \bwe \pt(\up a){\sm}\up p\nleq q}$. It must be the minimal element of this set, as for an element $q$ to be in this set we need to have $p\leq q$. Then, indeed, $p=\bwe \pt(\up a){\sm}\up p\ra a$ and so $p\in \bl(a)$. For the converse, suppose that we have $p\in \bl(a)$. Then, we have that $(p\ra a)\ra a=p$. This means that
\[
p=\bwe \set{(p\ra a)\ra q\mid q\in \pt(\up a)}
\]
The right hand side is again a conjunct of primes, and in particular it is 
\[
\bwe \set{q\in \pt(\up a)\mid p\ra a\nleq q}
\]
Now, let us observe that $p\ra a$ is exactly $\bwe \set{q\in \pt(\up a)\mid p\nleq q}=\bwe \pt(\up a){\sm}\up p$. Then, the expression above equals
\[
\bwe \set{q\in \pt(\up a)\mid \bwe \pt(\up a){\sm}\up p\nleq q}.
\]
Since this expression equals $p$, if we had $\bwe \pt(\up a)\leq p$, we would also have $\bwe \pt(\up a)\leq q$ for every element $q$ of the set $\set{q\in \pt(\up a)\mid \bwe \pt(\up a){\sm}\up p\nleq q}$, as $p$ is a lower bound for this set. This is a contradiction, by the membership condition of this set. The prime $p$ must then be essential.
\end{proof}

\begin{lemma}\label{boolandopen}
For a frame $L$ and for $x,y\in L$ we have that $\bl(x\ra y)=\op(x)\cap \bl(y)$.
\end{lemma}
\begin{proof}
Suppose that $z\in \bl(x\ra y)$. Then, we have that $z=z'\ra (x\ra y)$ for some $z'\in L$. Then, $z=(z'\we x)\ra y$ and so $z\in \bl(y)$. Additionally, we have $x\ra z=x\ra (z'\ra (x\ra y))=(x\we z')\ra (x\ra y)=(x\we z')\ra y=z'\ra (x\ra y)=z$. Then, $z\in \op (x)$. For the converse, suppose that $z\in \op(x)\cap \bl(y)$. Then, we have $x\ra z=z$ and also that $z=z'\ra y$ for some $z'\in L$. Then, $z=x\ra (z'\ra y)=x\we z'\ra y=z'\ra (x\ra y)$, and so $z\in \bl(x\ra y)$.
\end{proof}
\begin{lemma}\label{meetofabsess}
For a frame $L$, we have that it is spatial and every element other than $1$ has an absolutely essential prime above it if and only if every element is the meet of the essential primes above it.
\end{lemma}
\begin{proof}
If $L$ is a frame in which every element is the meet of the essential primes above it, in particular every element $a\in L$ with $a\neq 1$ must have at least an essential prime above it. The frame is also spatial. For the other implication, suppose that $L$ is spatial and that every element other than $1$ has an absolutely essential prime above it. Let $a\in L$ and let $\mf{AbsEss}(a)$ be the collection of absolutely essential primes of $a$. Suppose towards contradiction that $\bwe\mf{AbsEss}(a)\nleq a$. Then, $\bwe \mf{AbsEss}(a)\ra a\neq 1$ and so there is an absolutely essential prime $q$ of $\bwe \mf{AbsEss}(a)\ra a$. This means that we have $q\in \op (\bwe \mf{AbsEss}(a))\cap \bl(a)$ (see Lemma~\ref{booleans} and \ref{boolandopen}) and that $q$ is weakly covered. Since $q\in \op(\bwe \mf{AbsEss}(a))$ we have $q\notin \cl(\bwe \mf{AbsEss}(a))$, and so $\bwe \mf{AbsEss}(a)\nleq q$. Since $q\in \bl(a)$ and so it is an absolutely essential prime of $a$. This contradicts that $\bwe \mf{AbsEss}(a)\nleq q$.
\end{proof}

We are ready to prove the main theorems of this section.

\begin{theorem}\label{m1}
For a frame $L$, the following are equivalent.
\begin{enumerate}[\normalfont(1)]
    \item All sublocales of $L$ are $T_D$-spatial;
    \item All \donotbreakdash{$D$}sublocales of $L$ are \donotbreakdash{$T_D$}spatial;
    \item $\mathfrak{M}\circ \pt_D$ is the identity on $\mf{S}_D(L)$;
    \item There is an isomorphism  $\SSS_D(L)\cong \mathcal{P}(\pt_D(L))$;
    \item $\SSS_D(L)$ is spatial and Boolean (i.e. a complete and atomic Boolean algebra);
    \item $\SSS_D(L)=\SSS_b(L)$ and $L$ is \donotbreakdash{$T_D$}spatial.
    \item Every nonzero sublocale of $L$ contains a covered prime in itself.
\end{enumerate}

\end{theorem}
\begin{proof}(1) obviously implies (2).\\[2mm]
(2) is equivalent to (3). This holds by Lemma~\ref{tdsub}\\[2mm]
(3)$\implies$(4). The isomorphism is obtained by composing $\mathfrak{M}\circ \pt_D$ and the isomorphism $(\mathfrak{M}\circ \pt_D)[\SLd]\ra (\pt_D\circ \mathfrak{M})[\ca{P}(\pt_D(L))]$.\\[2mm]
(4)$\implies$(5). This is clear.\\[2mm]
(5)$\implies$(6). Suppose that $\SLd$ is spatial and Boolean. 
If $\SSS_D(L)$ is Boolean, since it is also dense, it must coincide with the unique Boolean dense sublocale of $\SSS(L)$, namely $\SSS_b(L)$. Now, if $\SSS_D(L)=\SSS_b(L)$ is also spatial, by Theorem~\ref{ss} we conclude that $L$ is \donotbreakdash{$T_D$}spatial.\\[2mm]
(6) $\implies$ (1). If we have $\SLb=\SLd$, and if $L$ is \donotbreakdash{$T_D$}spatial, by Theorem~\ref{ss} we have that $\SLd$ is spatial, and so by Proposition~\ref{locandglob} all \donotbreakdash{$D$}sublocales are \donotbreakdash{$T_D$}spatial.\\[2mm]
(6) $\implies$ (7). According to Proposition~\ref{charDsc}, every non-smooth sublocale has a covered prime in itself. But $L$ being \donotbreakdash{$T_D$}spatial, say $L=\Omega(X)$ with $X$ a \donotbreakdash{$T_D$}space, every (nonzero) smooth sublocale of it is induced by a subspace $A\subseteq X$ and hence it contains a covered prime of $L$, in particular a covered prime in itself. Hence every nonzero sublocale contains a covered prime in itself.\\[2mm]
(7)$\implies$(1). Clearly, the condition (7) is hereditary with respect to any sublocale. Hence, it suffices to show that $L$ is  \donotbreakdash{$T_D$}spatial.  If for all complemented sublocale $C\subseteq L$ such that $\spa_D(L)\subseteq C$ one has $C=L$, then by zero-dimensionality of $\SSS(L)^{op}$, it follows that $L=\spa_D(L)$ is $T_D$-spatial. Hence assume by way of contradiction that there is a complemented sublocale $C\subseteq L$ such that $C\ne L$ and $\spa_D(L)\subseteq C$.  Since $C\ne L$, then $C^c\ne 0$, so by assumption there is a $p\in \pt_D(C^c)$ with $ \bl(p)\subseteq C^c$. But $C^c$ is complemented; in particular it is a $D$-sublocale, hence $p\in \pt_D(C^c)\subseteq \pt_D(L)$.  But then $\bl(p)\subseteq \spa_D(L)\subseteq C$, whence $\bl(p)\subseteq C\cap C^c=0$, contradiction.
\end{proof}

We now move on to looking at the stronger condition of a frame being totally spatial and \emph{strongly} $T_D$-spatial. To see that this is condition is not equivalent to those in Theorem \ref{m1} above, consider Example 7.8 in  \cite{avila2019frame}.

\begin{theorem}\label{m22}
For a frame $L$, the following are equivalent.
\begin{enumerate}[\normalfont(1)]
    \item The frame $L$ is totally spatial and all its primes are covered;
    \item  The frame $L$ is totally spatial and strongly \donotbreakdash{$T_D$}spatial;
        \item All sublocales of $L$ are strongly \donotbreakdash{$T_D$}spatial;
    \item There is an isomorphism $\SSS(L)\cong \mathcal{P}(\pt_D(L))$;
    \item The coframe $\SL$ is spatial and Boolean;
    \item Every element of $L$ is the meet of the covered essential primes above it;
      \item Every element of $L$ is the meet of the covered absolutely essential primes above it;
    \item The frame $L$ is spatial and every element other than $1$ has a covered absolutely essential prime above it;
    \item Every nonzero sublocale of $L$ contains a covered prime of $L$.
\end{enumerate}
\end{theorem}

\begin{proof}
(1)$\implies$(2). If $L$ is totally spatial, then in particular it is spatial; if additionally all its primes are covered the frame is also strongly \donotbreakdash{$T_D$}spatial.\\[2mm] 
(2)$\implies$(3). If $L$ is totally spatial, then every sublocale is spatial, so every element of $S$ is a meet of primes for each $S\in \SSS(L)$. But every $p\in\pt(S)$ is a covered prime in $L$, and in particular in $S$, so every element of $S$ is a meet of covered primes in $S$. Now use Lemma~\ref{me}.\\[2mm]
(3)$\implies$(4). If all sublocales of $L$ are strongly \donotbreakdash{$T_D$}spatial, in particular all of the primes of $L$ are covered and so we have $\SLd=\SL$. Furthermore, all of the sublocales being strongly \donotbreakdash{$T_D$}spatial means, by commutativity of the main square of this section, that we have an isomorphism $\SLd\cong \ca{P}(\pt_D(L))$.\\[2mm] 
(4)$\implies$(5). This is clear.\\[2mm]
(5)$\implies$(6). If $\SL$ is spatial, the frame $L$ is totally spatial, and so every element is the meet of its essential primes by Lemma~\ref{totspaess}. Furthermore, if $\SL$ is Boolean, in particular all two-element sublocales are complemented, and so every prime is covered.\\[2mm]
(6)$\implies$(7). This follows from Lemma~\ref{essentialprimes} (if a prime is essential for one element and it is covered, it is also absolutely essential).\\[2mm]
(7)$\implies$(8). If all elements are the meets of the absolutely essential primes above it, then in particular this must be true for elements other than $1$, and so their corresponding meet of absolutely essential primes must be nonempty.\\[2mm]
(8) and (9) are equivalent. If (8) holds, then by Lemma~\ref{booleans} all nonzero Boolean sublocales have at least a point and since all sublocales are unions of Boolean sublocales this also implies that all nonzero sublocales have at least a point. For the converse, suppose that every nonzero sublocale of $L$ has a covered prime. In particular, every Boolean sublocale $\bl(a)$ with $a\neq 1$ has a prime $p\in \bl(a)$, and this by Lemma~\ref{booleans} means that $a$ has an essential prime which is covered. But since $p$ is covered, by Lemma~\ref{coveredess} this implies that it is also an absolutely essential prime of $a$. \\[2mm] 
(8)$\implies$(1). Suppose that $L$ is spatial, and that every element other than $1$ has a covered absolutely essential prime above it. By Lemma~\ref{meetofabsess}, we know that the frame is totally spatial, as every element is the meet of the absolutely essential primes above it and so, a fortiori, it is the meet of the essential primes above it. Let us show that all primes of $L$ are covered. Every prime $p$ has an absolutely essential prime and so, in particular, the meet $\bwe \{p\}$ has an essential prime, which must be $p$ itself. Furthermore, our assumption also tells us that the absolutely essential prime $p$ must be covered. 
\end{proof}

\section{The relation between $\SLd$ and $\spa[\SL]$}

We begin this section with a non-spatial version of Theorem~\ref{ssss}, that is, we explore locales in which every prime is covered. Note that this is the $T_D$ analogue of the point-free notion of $T_1$ locale introduced by Rosický and Šmarda \cite{T1loc}. 
\begin{theorem}\label{prcov}
For a frame $L$, the following are equivalent.
\begin{enumerate}[(1)]
    \item All primes of $L$ are covered;
    \item $\spa[\mf{S}(L)]\se \mf{S}_b(L)$;
 \item $\SSS_D(L)=\SSS(L)$;
 \item $\SSS_D(L)$ is closed under arbitrary intersections in $\SSS(L)$;
   \item $\spa[\SSS(L)]\subseteq \SSS_D(L)$.
\end{enumerate}
\end{theorem}
\begin{proof}
Suppose that (1) holds. Then, (2) holds by Corollary \ref{a}.  Suppose that (2) holds and let $S$ an arbitrary sublocale. If $p\in \pt_D(S)$, then in particular $p\in \pt(S)$ and so $\mathfrak{b}(p)\in \spa[\SSS(L)]\subseteq \SSS_b(L)$. By Lemma~\ref{a'}, one has that $p$ is covered in $L$, and so $S\in \SSS_D(L)$. That (3) implies (4) is obvious and (4) implies (3) because every sublocale is an intersection of complemented sublocales (and hence of $D$-sublocales). Assume now that (3) holds and let $p$ be a prime. Then $\mathfrak{b}(p)\in \spa[\SSS(L)]\subseteq\SSS_D(L)$ and so $p\in \pt_D(\mathfrak{b}(p))\subseteq \pt_D(L)$. Hence (1) follows.
\end{proof}

We can use $\SSS_D(L)$ for characterizing a number of well-known properties for a frame $L$. For example we have the following (but see also Theorem~\ref{m1} or Theorem~\ref{prcov}).

\begin{proposition}
For a frame $L$, the following are equivalent:
\begin{enumerate}[\normalfont(1)]
\item $L$ is totally spatial\textup;
\item $\SSS_D(L)\subseteq \spa[\SSS(L)]$\textup.
\end{enumerate}
\end{proposition}

\begin{proof}
(1)$\implies$(2). This is trivial since if $L$ is totally spatial then $\spa[\SSS(L)]=\SSS(L)$.\\[2mm]
(2)$\implies$(1). By the observation preceding the proposition, every pointless sublocale of $L$ is in $\SSS_D(L)$, and so by assumption, every pointless sublocale is spatial. Thus $L$ has no nontrivial pointless sublocales, i.e. every nontrivial sublocale of $L$ has a point. By \cite[p.~269]{niefield87}, it follows that $L$ is totally spatial. \end{proof}

Next we explore the inclusion  $\SSS_D(L)\subseteq \SSS_b(L)$ (note that the inclusion is actually equivalent to the equality $\SSS_D(L)=\SSS_b(L)$, since every frame has a unique Boolean and dense sublocale). We first give a necessary condition for it:

\begin{proposition}\label{TSifSC}
If  $\SSS_D(L)\subseteq \SSS_b(L)$ then $\spa(L)$ is totally spatial.
\end{proposition}

\begin{proof}
Let $S$ be a pointless sublocale contained in $\spa(L)$.  Then $S\in \SSS_D(L)\subseteq \SSS_b(L)$. 
 It is easy to show  that  one always has  $\SSS_b(\spa(L))\supseteq \downarrow^{\SSS_b(L)}\spa(L)$. It follows that $S\in \SSS_b(\spa(L))$. But $\spa(L)$ is spatial and so $S\in \SSS_b(\spa(L))\subseteq \spa[\SSS(\spa(L))]$, but a spatial pointless sublocale must be trivial. Hence $\spa(L)$ does not contain any nontrivial pointless sublocales, which by \cite[p.~269]{niefield87} implies that $\spa(L)$ is totally spatial.
\end{proof} 

We now characterize frames for which the inclusion $\SSS_D(L)\subseteq \SSS_b(L)$ holds. Despite the fact that we have not been able to find a clear geometric condition as those occurring in other results of the paper (spatiality, total spatiality, scatteredness, \donotbreakdash{$T_D$}spatiality,\dots), condition (2) below indeed tells us something: the assumption in (1) (i.e. that a given sublocale $S$  is a  \donotbreakdash{$D$}sublocale) is relative to $L$, whereas the assumption in (2) is absolute, in the sense that it does not depend on the ambient frame $L$:

\begin{proposition}\label{charDsc}
For a frame $L$, the following are equivalent:
\begin{enumerate}[\normalfont(1)]
\item $\SSS_D(L)\subseteq \SSS_b(L)$ (i.e. every \donotbreakdash{$D$}sublocale is smooth);
\item Sublocales of $L$ without covered primes in themselves are smooth in $L$.
\end{enumerate}
\end{proposition}

\begin{proof}
(1)$\implies$(2). Let $S$ be a sublocale such that $\pt_D(S)=\varnothing$. Then $S\in \SSS_D(L)\subseteq \SSS_b(L)$ so it is smooth.\\[2mm]
(2)$\implies$(1). Let $S\in \SSS_D(L)$, and consider the decomposition $S=\spa_D(S)\vee (S\setminus \spa_D(S))$. Observe that by Proposition~\ref{impr}, $S\setminus \spa_D(S)$ is a \donotbreakdash{$D$}sublocale as well, i.e. $\pt_D(S\setminus \spa_D(S))\subseteq \pt_D(L)$. Assume that there exists some $p\in \pt_D( S\setminus \spa_D(S) ).$ Then $p\in \pt_D(L)\cap S\subseteq \pt_D(S)$ and hence $\bl(p)\subseteq \spa_D(S)$. Therefore, $\bl(p)\subseteq S\setminus \spa_D(S)\subseteq S\setminus \bl(p)\subseteq L\setminus \bl(p)= \bl(p)^c$, which yields a contradiction. Hence $\pt_D(S\setminus \spa_D(S))=\varnothing$ and so $S\setminus \spa_D(S)$ is smooth in $L$. Moreover, $\spa_D(S)=\bigvee_{p\in \pt_D(S)}\bl(p)$ is also a smooth sublocale because $\pt_D(S)\subseteq \pt_D(L)$ and Lemma~\ref{a'}. Hence $S=\spa_D(S)\vee (S\setminus \spa_D(S))$ is smooth, as it is a join of smooth sublocales.
\end{proof}

By analogy with the classical spectrum, frames $L$ for which $\SSS_D(L)\subseteq \SSS_b(L)$ holds are going to be called \emph{\donotbreakdash{$D$}scattered}.

\section{Adding $\mf{S}_{c}(L)$ to the picture}

First, let us connect the frame of joins of closed sublocales with the spatialization sublocale of $\SL$.
\begin{lemma}\label{t11}
A frame $L$ is spatial if and only if $\SLc\se \spa[\SL]$.
\end{lemma}
\begin{proof}
Suppose that $L$ is spatial. For a closed sublocale $\up a\se L$, suppose that we have $b\in \up a$. By spatiality, $b=\bwe \pt(\up b)$, but since all primes above $b$ are also above $a$, this is the same as $\bwe (\pt(\up a)\cap \up a)=\bwe \set{p\in \up a\mid b\leq p}$. Then, $\up a\se \mathfrak{M}(\pt(\up a))$, and this means that $\up a$ is spatial. Conversely, if $\SLc\se \spa[\SL]$ we have that every closed sublocale is spatial, and in particular $L$ itself is. 
\end{proof}

\begin{lemma}\label{t12}
A frame $L$ is such that all its primes are maximal if and only if $\spa[\SL]\se \SLc$.
\end{lemma}
\begin{proof}
Suppose that for a frame $L$ all its primes are maximal. This means that for every prime $p\in \pt(L)$ we have that $\bl(p)$ is closed, and since the elements of $\spa[\SL]$ are all joins of elements of the form $\bl(p)$, we have $\spa[\SL]\se \SLc$. Conversely, if we have a prime $p\in L$ which is not maximal, say $p<x<1$, we have that $\bl(p)\neq \cl(p)$, and since $\bl(p)$ does not equal its closure it is not closed. Then, $\spa[\SL]\nsubseteq \SLc$.
\end{proof}

Secondly, we look at the relation between $\SLc$ and $\SLd$. First, we observe that we always have $\SLc\se \SLd$ (we know that all closed sublocales are \donotbreakdash{$D$}sublocales, and the collection of \donotbreakdash{$D$}sublocales is closed under joins).

\begin{proposition}
For a frame $L$, we have that $\SLd\se \SLc$ if and only if $L$ is subfit and \donotbreakdash{$D$}scattered.
\end{proposition}
\begin{proof}
Suppose that for a frame $L$ we have $\SLd\se \SLc$. We have that all closed sublocales are complemented, and so $\SLc\se \SLb$. We then have that $\SLd\se \SLb$, which means that $L$ is \donotbreakdash{$D$}scattered. Furthermore, as this implies also that $\SLd=\SLb$, we also have that $\SLc=\SLb$, which means that $L$ is subfit (see \cite{picado19}). Conversely, suppose that $L$ is \donotbreakdash{$D$}scattered and subfit. Since $L$ is subfit, we have that $\SLc=\SLb$ (see \cite{picado19}), and since $L$ is \donotbreakdash{$D$}scattered, we have that $\SLd=\SLb$. Hence, $\SLd=\SLc$.
\end{proof}

\bibliographystyle{acm}
\bibliography{biblioTD}
\end{document}